\documentclass[11pt]{article}

\usepackage[margin=1in]{geometry}
\usepackage{color}
\usepackage{amssymb}
\usepackage{amsmath}
\usepackage{amsthm}
\usepackage{url}
\usepackage{amsfonts}
\usepackage{hyperref}
\usepackage{graphicx}
\usepackage{tikz}

\newtheorem{theorem}{Theorem}
\newtheorem{definition}[theorem]{Definition}
\newtheorem{proposition}[theorem]{Proposition}
\newtheorem{lemma}[theorem]{Lemma}
\newtheorem{corollary}[theorem]{Corollary}
\newtheorem{conjecture}[theorem]{Conjecture}

\begin{document}

\title{Packings in real projective spaces}

\author{Matthew Fickus\footnote{Department of Mathematics and Statistics, Air Force Institute of Technology, Wright-Patterson AFB, OH} \qquad John Jasper\footnote{Department of Mathematics and Statistics, South Dakota State University, Brookings, SD} \qquad Dustin G.\ Mixon\footnote{Department of Mathematics, The Ohio State University, Columbus, OH}}
\date{}

\maketitle

\begin{abstract}
This paper applies techniques from algebraic and differential geometry to determine how to best pack points in real projective spaces.
We present a computer-assisted proof of the optimality of a particular 6-packing in $\mathbb{R}\mathbf{P}^3$, we introduce a linear-time constant-factor approximation algorithm for packing in the so-called Gerzon range, and we provide local optimality certificates for two infinite families of packings.
Finally, we present perfected versions of various putatively optimal packings from Sloane's online database, along with a handful of infinite families they suggest, and we prove that these packings enjoy a certain weak notion of optimality.
\end{abstract}

\section{Introduction}

Given a compact metric space, it is natural to ask how to pack $n$ points so that the minimum distance is maximized.
According to legend, packing in the unit 2-sphere with the great-circle distance incited a dispute in 1694 between Isaac Newton and David Gregory~\cite{Casselman:04}.
More recently, packing in $\{0,1\}^k$ with the Hamming distance has produced codebooks that form the foundation of error correction in digital communication~\cite{ClarkC:81}.

Motivated by optimal tumor treatment with high-energy laser beams, Conway, Hardin and Sloane in 1996 were the first to study packings in Grassmannian spaces with the chordal distance~\cite{ConwayHS:96}.
The past decade has seen a surge of work in the special case of projective spaces due to applications in multiple description coding~\cite{StrohmerH:03}, digital fingerprinting~\cite{MixonQKF:13}, compressed sensing~\cite{BandeiraFMW:13}, and quantum state tomography~\cite{RenesBSC:04}.
Developments in this packing problem have largely built on the precursor work of Rankin~\cite{Rankin:55,Rankin:56}, Grey~\cite{Grey:62}, Seidel~\cite{Seidel:73} and Welch~\cite{Welch:74}, the vast majority of advances leveraging ideas from combinatorial design; see \cite{FickusM:15} for a survey.

In this paper, we apply techniques from algebraic and differential geometry to obtain new packing results over real projective spaces.
Our intent is to introduce new methods for tackling these packing problems while simultaneously introducing the problems to the broader mathematical community.
In particular, the authors believe there is ample opportunity for algebraic geometers to make significant contributions to this burgeoning research program.

\subsection{Preliminaries}

It is convenient to identify unit vectors with the lines that they span.
In this spirit, we define an \textbf{$n$-packing} in $\mathbb{R}\mathbf{P}^{d-1}$ to be a sequence of $n$ unit vectors in $\mathbb{R}^d$.
Throughout, we will use $\Phi$ to denote both the sequence $\{\varphi_i\}_{i\in[n]}$ in $\mathbb{R}^d$ and the $d\times n$ matrix whose $i$th column vector is $\varphi_i$; here, $[n]:=\{1,\ldots,n\}$, with the understanding that $[0]$ corresponds to the empty set.
Given two unit vectors $x$ and $y$, the chordal distance between $\operatorname{span}\{x\}$ and $\operatorname{span}\{y\}$ is a decreasing function of the \textbf{correlation} $|\langle x,y\rangle|$; thanks to this relationship, our representation of points in $\mathbb{R}\mathbf{P}^{d-1}$ by unit vectors in $\mathbb{R}^d$ is not problematic.
The \textbf{coherence} of an $n$-packing $\Phi=\{\varphi_i\}_{i\in[n]}$ is given by
\[
\mu(\Phi):=\max_{\substack{i,j\in[n]\\i\neq j}}|\langle \varphi_i,\varphi_j\rangle|.
\]
An $n$-packing $\Phi$ in $\mathbb{R}\mathbf{P}^{d-1}$ is \textbf{optimal} if $\mu(\Phi)\leq\mu(\Psi)$ for all $n$-packings $\Psi$ in $\mathbb{R}\mathbf{P}^{d-1}$; such packings necessarily exist by compactness.
Optimal packings are called \textit{Grassmannian frames} in the finite frames literature~\cite{StrohmerH:03}, but we avoid this terminology here for the sake of clarity.

Optimal packings in the $d=2$ case are trivial: $\mathbb{R}\mathbf{P}^1$ is isometric to the circle, and so optimal packings amount to regularly spaced points by the pigeonhole principle~\cite{BenedettoK:06}.
Packing is more difficult in higher dimensions.
In~\cite{Welch:74}, Welch introduced a useful lower bound:
\begin{equation}
\label{eq.welch proof}
0
\leq\Big\|\Phi\Phi^\top-\frac{n}{d}I\Big\|_F^2
=\|\Phi^\top\Phi\|_F^2-\frac{n^2}{d}
\leq n+n(n-1)\mu(\Phi)^2-\frac{n^2}{d},
\end{equation}
where rearranging gives
\begin{equation}
\label{eq.welch bound}
\mu(\Phi)
\geq\sqrt{\frac{n-d}{d(n-1)}}.
\end{equation}
By inspecting the proof \eqref{eq.welch proof} of Welch's bound, one observes that equality occurs precisely when both $\Phi\Phi^\top=(n/d)I$ and all of the off-diagonal entries of the Gram matrix $\Phi^\top\Phi$ are $\pm\mu(\Phi)$.
Packings that satisfy $\Phi\Phi^\top=(n/d)I$ are known as \textbf{tight frames}, whereas the second condition establishes equiangularity between the packing members.
Tight frames, corresponding to \textit{eutactic stars} in Euclidean geometry~\cite{Schlafli:49}, are particularly interesting because they provide redundant versions of orthonormal bases, as evidenced by the Parseval identity they satisfy:
\[
\sum_{i\in[n]}|\langle x,\varphi_i\rangle|^2
=\alpha\|x\|^2
\qquad
\forall x\in\mathbb{R}^d
\]
for some $\alpha>0$.
Tight frames have the defining property that the row vectors of $\Phi$ are orthogonal with equal norm, and completing these rows to an equal-norm orthogonal basis produces the row vectors of an $(n-d)\times n$ tight frame called a \textbf{Naimark complement} of $\Phi$.
If $\Phi$ is tight with unit vectors, we may scale its Naimark complement to produce another tight frame $\Psi$ with unit vectors, and the coherences of these packings are related:
\[
\frac{d}{n}\Phi^\top\Phi+\frac{n-d}{n}\Psi^\top\Psi
=I
\qquad
\Longrightarrow
\qquad
\frac{d}{n}\mu(\Phi)-\frac{n-d}{n}\mu(\Psi)=0.
\]

Overall, the (optimal) packings that achieve equality in the Welch bound are the so-called \textbf{equiangular tight frames (ETFs)}~\cite{FickusM:15}.
ETFs are in one-to-one correspondence with a subclass of strongly regular graphs~\cite{Waldron:09}.
In particular, given an ETF, negate packing elements so that they all have positive inner product with the last element, and then remove this last element to produce a $(n-1)$-subpacking $\Psi$; the sign pattern of the Gram matrix $\Psi^\top\Psi$ then corresponds to the adjacency matrix of a strongly regular graph on $n-1$ vertices.
As such, the plethora of strongly regular graphs tabulated by Brouwer in~\cite{Brouwer:online} lead to immediate solutions to the packing problem.

Unfortunately, ETFs only exist for certain choices of $(d,n)$.
For instance, in the nontrivial case where $1<d<n-1$, we know that $n$ must lie in the \textbf{Gerzon range}, defined by
\begin{equation}
\label{eq.gerzon}
d+\sqrt{2d+\tfrac{1}{4}}+\tfrac{1}{2}
\leq n
\leq \tfrac{1}{2}d(d+1).
\end{equation}
Indeed, the upper bound follows from the linear independence of the outer products $\{\varphi_i\varphi_i^\top\}_{i\in[n]}$ (hint:\ consider their Gram matrix), and the lower bound then follows from the Naimark complement.
Furthermore, $(d,n)$ must satisfy certain integrality conditions~\cite{FickusM:15} that make it exceedingly difficult for $(d,n)$ to admit an ETF.
These restrictions have prompted researchers to investigate alternatives to the Welch bound.
To date, the most fruitful approach along these lines has been to consider $n$ beyond the Gerzon range, where alternatives like the orthoplex bound, the Levenshtein bound, and more generally, Delsarte's linear programming bound begin to take hold~\cite{BodmannH:16,HaasHM:17}.
In this regime, the provably optimal packings tend to be tight frames with small angle sets.
This has spurred interest in so-called \textbf{biangular} frames~\cite{HaasCTC:17}, in which the off-diagonal Gram matrix entries all have the form $\pm\alpha$ and $\pm\beta$ for some $\alpha,\beta\geq0$.
As we will see, biangular packings with $\alpha=0$ emerge frequently in practice, in which case we call the packing \textbf{orthobiangular}.

\subsection{Motivating applications}

There are a number of applications that leverage ensembles of unit vectors to encode or decode some sort of signal.
For each of these applications, performance is a function of the coherence of the ensemble, with smaller coherence resulting in better performance.
We review a few of these applications below.

\subsubsection{Compressed sensing}

Given an $n\times n$ orthogonal matrix $Q$, consider vectors of the form $u=Qx$, where $x$ has at most $k$ nonzero entries; that is, $u$ is \textbf{$k$-sparse} in $Q$.
For example, natural images tend to be well-approximated by vectors which are sparse in the adjoint of the discrete wavelet transform, and this feature is exploited in the JPEG 2000 compression standard~\cite{JPEG2000}.
The goal of compressed sensing is to leverage this sparsity in order to solve appropriately selected underdetermined linear systems~\cite{CandesRT:06,Donoho:06}.
As an application of this theory, compressive MRI scans require a fraction of acquisition time over conventional MRI~\cite{LustigDSP:08}.

Let $A$ be $d\times n$ with $d$ much smaller than $n$.
To reconstruct $u$ from measurements of the form $y=Au$, we write $\Phi=AQ$.
Then solving for $u$ is equivalent to the following non-convex program:
\[
\text{find}
\quad
x
\quad
\text{subject to}
\quad
y=\Phi x,
\quad
|\operatorname{supp}(x)|\leq k.
\]
Provided $k\leq d/2$, then $Q^{-1}u$ is the unique solution to the above program for almost every choice of $A$, in which case it suffices to find the smallest $k$ for which the above program is feasible.
This suggests the following convex relaxation, which can be solved with linear programming:
\[
\text{minimize}
\quad
\|x\|_1
\quad
\text{subject to}
\quad
y=\Phi x.
\]
Impressively, this relaxation solves the original non-convex program when the columns of $\Phi$ have unit norm and small coherence $\mu(\Phi)$.
In particular, there exist universal constants $c_1,c_2>0$ such that the relaxation exactly recovers all $k$-sparse vectors $x$ with $k\leq c_1\mu(\Phi)^{-1}$~\cite{DonohoE:03}, as well as most $k$-sparse vectors with $k\leq c_2\mu(\Phi)^{-2}/\log n$~\cite{Tropp:08}\footnote{Beware that in~\cite{Tropp:08}, the $\log(N/\delta)$ that appears in Model (M1) should be $1/\log(N/\delta)$, as can be verified by writing out the last line of the proof of Theorem~14.}.
For this last statement, ``most'' can be replaced by ``all'' when $\Phi$ has independent Gaussian entries~\cite{BaraniukDDW:08}, in which case $k$ can be as large as $O(d/\log n)$.
No explicit matrix is known to exhibit this behavior~\cite{BandeiraFMW:13}.
To date, the best explicit matrices enable exact recovery whenever $k\leq c_3d^{1/2+\epsilon}$ for some small $c_3,\epsilon>0$~\cite{BourgainDFKK:11,BandeiraMM:17}; in both cases, the matrices are constructed from optimal packings in projective spaces.

\subsubsection{Digital fingerprinting}

Suppose a content owner wishes to distribute a file to a specific list of recipients, but also wants to identify leakers if the file is disclosed to additional recipients.
Let $x\in\mathbb{R}^d$ denote the original file, and let $n$ denote the number of intended recipients.
Then the content owner can transmit a slightly personalized version of $x$ to each of the recipients.
For example, given a collection of \textbf{fingerprints} $\{\varphi_i\}_{i\in[n]}$ in $\mathbb{R}^d$, suppose the $i$th user receives $y_i=x+c\varphi_i$ for some $c>0$.
Here, $c$ is chosen to be small enough so that each $y_i$ is subjectively true to the original $x$ (i.e., the user's experience in enjoying $y_i$ isn't disturbed by the inclusion of $\varphi_i$), but large enough so that the fingerprint will effectively incriminate the $i$th user in the event that, say, $y_i$ becomes popular on the internet.

In order to combat this hinderance to piracy, we envision an attack in which multiple users $S\subseteq[n]$ collude to produce a forgery:
\[
\hat{x}
=\sum_{i\in S}a_iy_i+e
=x+\sum_{i\in S}a_i\varphi_i+e,
\]
where the second equality requires the weights $\{a_i\}_{i\in[n]}$ to sum to $1$, and $e\sim N(0,\sigma^2I)$ is Gaussian noise.
Here, $\sigma^2$ is chosen to be small enough to be subjectively true to the original $x$, but large enough to hopefully cover the culprits' tracks.
Now suppose the content owner encounters the forgery $\hat{x}$.
Then by subtracting the original $x$, he isolates a noisy combination of fingerprints $\{\varphi_i\}_{i\in S}$.
At this point, he may identify the fingerprint that correlates most with this combination:
\[
j:=\arg\max_{i\in[n]}|\langle \hat{x}-x,\varphi_i\rangle|.
\]
Notice that $j$ is a random variable due to $e$.
It turns out that $j\in S$ with large probability provided the coherence $\mu(\Phi)$ of the fingerprints is small~\cite{MixonQKF:13}.
Furthermore, once one of the colluders is identified, the others can be identified through the legal process.

\subsubsection{Quantum state tomography}

The goal of quantum state tomography is to estimate the state of a quantum mechanical system, modeled as a self-adjoint, positive semidefinite $d\times d$ matrix $\rho$ of unit trace.
We consider measurements performed in terms of a discrete \textbf{positive operator--valued measure (POVM)}, that is, a sequence $\{F_i\}_{i\in[n]}$ of self-adjoint, positive semidefinite $d\times d$ matrices that sum to the identity matrix.
When measuring $\rho$ with a POVM $\{F_i\}_{i\in[n]}$, the outcome is a random variable $X$ taking values in $[n]$ with probabilities given by the Born rule
\[
\operatorname{Pr}(X=i)=\operatorname{tr}(\rho F_i).
\]
These probabilities can be approximated from empirical frequencies after sufficiently many measurements, and then one can estimate $\rho$ by solving a linear system, provided $\{F_i\}_{i\in[n]}$ spans the $d^2$-dimensional real vector space of self-adjoint $d\times d$ matrices.
Such POVMs are called \textbf{informationally complete (IC)}.

In the minimal case where $n=d^2$, if we insist that $\operatorname{tr}(F_i)$ and $\|F_i\|_F$ both be constant over $i\in[n]$, then the linear inverse problem is best conditioned when each of the $F_i$'s has the form $F_i=(1/d)\varphi_i\varphi_i^*$ with $\mu(\Phi)=1/\sqrt{d+1}$, provided there exists an ensemble $\Phi=\{\varphi_i\}_{i\in[n]}$ with these specifications.
Such IC-POVMs are called \textbf{symmetric (SIC-POVMs, or SICs)}.
Considering $1/\sqrt{d+1}$ equals the Welch bound for $n=d^2$, SICs are necessarily optimal packings in $\mathbb{C}\mathbf{P}^{d-1}$.
This particularly natural choice of POVM has been proposed as a standard quantum measurement, and lies at the foundation of quantum Bayesianism~\cite{FuchsS:11}.

Despite being such natural mathematical objects, SICs are only known to exist for finitely many dimensions $d$~\cite{Flammia:online}, though they are conjectured to exist for every $d\geq2$.
In fact, Zauner~\cite{Zauner:99} conjectures that for every $d\geq2$, there exists an eigenvector of a certain order-$3$ unitary operator whose orbit under the Heisenberg--Weyl group forms a SIC.
To date, there is numerical evidence in favor of Zauner's conjecture for all $d\leq151$~\cite{FuchsHS:17}, and recent work suggests that a constructive proof of this conjecture may require progress on Hilbert's twelfth problem~\cite{ApplebyFMY:17}.

\subsubsection{Multiple description coding}

Suppose Alice wishes to transmit a message to Bob, but the channel through which she must communicate will corrupt the message.
How can Alice make sure that Bob receives the intended message?
We consider an \textbf{erasure channel}, which is modeled as follows:
Given a transmitted vector $y\in\mathbb{R}^n$, the received vector has the form $E(y+e)$, where $e\sim N(0,\sigma^2I)$ is Gaussian noise and $E$ is some diagonal matrix of zeros and ones, thereby erasing certain entries.
In practice, we can predict the number of erasures in $E$, but not the locations of these erasures.

In order for Bob to have enough information to reconstruct Alice's message, Alice needs to redundantly encode her message before transmitting through the erasure channel.
To this end, we consider linear encodings, where a message of the form $x\in\mathbb{R}^d$ is encoded as $\{\langle x,\varphi_i\rangle\}_{i\in[n]}$ for some ensemble $\{\varphi_i\}_{i\in[n]}$.
In other words, Alice transmits $y=\Phi^*x$, and so Bob receives $z=E(\Phi^*x+e)$.

In principle, Bob can compute the maximum likelihood estimator for $x$ by solving a least-squares problem.
Indeed, with probability $1$, the non-erased entries of $\Phi^*x+e$ correspond to the support $S$ of $z$, and so given $z$, Bob can restrict to $S$ to deduce $z_i=\langle x,\varphi_i\rangle+e_i$ for each $i\in S$.
Letting $\Phi_S$ denote the submatrix of $\Phi$ with columns in $S$, then the least-squares solution is given by $(\Phi_S\Phi_S^*)^{-1}\Phi_Sz$.
Unfortunately, computing this estimate requires Bob to invert a large matrix, leading to a prohibitively long runtime.
As a fast alternative, Bob ignores which entries of $z$ were erased and instead takes the estimate $\hat{x}=(\Phi\Phi^*)^{-1}\Phi x$.
Indeed, since the decoding matrix $\Phi^\dagger=(\Phi\Phi^*)^{-1}\Phi$ no longer depends on $S$, it can be computed in advance rather than on the fly.

Overall, Alice encodes $x$ as $y=\Phi^*x$ before transmitting, and upon receipt of $z=E(y+e)$, Bob decodes with $\hat{x}=\Phi^\dagger z$.
To evaluate the quality of reconstruction under $k$ erasures, we compute the worst-case mean squared error:
\[
\operatorname{MSE}_k
:=\max_{\substack{S\subseteq[n]\\|S|=n-k}}\max_{\substack{x\in\mathbb{R}^d\\\|x\|=1}}\mathbb{E}\Big[~\|\hat{x}-x\|^2~\Big|~S~\Big].
\] 
Here, $S$ denotes the locations of the nonzero diagonal entries in $E$, and the expectation is taken over the distribution of $e$.
Notice that Alice can theoretically diminish the effect of the additive noise $e$ by multiplying $\Phi$ by an arbitrarily large scalar.
However, power limitations preclude this, and so we fix its Frobenius norm for the analysis; say, $\|\Phi\|_F^2=n$.
Subject to this constraint, then $\operatorname{MSE}_0$ is minimized when $\Phi$ is a tight frame; this follows from the proof of Theorem~3.1 in~\cite{GoyalKK:01}.
Next, restricting to tight frames, $\operatorname{MSE}_1$ is minimized precisely by tight frames comprised of unit vectors (see Proposition~2.1 in~\cite{HolmesP:04}, cf.~\cite{CasazzaK:03}).
As a dual result, the ensembles of unit vectors that minimize $\operatorname{MSE}_1$ are tight frames (Theorem~4.4 in~\cite{GoyalKK:01}).
For unit norm tight frames, $\operatorname{MSE}_2$ is an increasing function of the coherence $\mu(\Phi)$ by the proof of Proposition~2.2 in~\cite{HolmesP:04}.
In the case where $\Phi$ is comprised of unit vectors, but is not necessarily tight, Section~4 of~\cite{StrohmerH:03} estimates $\operatorname{MSE}_k$ in terms of the coherence $\mu(\Phi)$; for this analysis, one decodes with a scalar multiple $\Phi$ instead of $\Phi^\dagger$ (when $\Phi$ is tight, $\Phi^\dagger=(d/n)\Phi$).
Overall, Bob can quickly and reliably reconstruct the intended message provided Alice encodes it with an optimal packing in $\mathbb{R}\mathbf{P}^{d-1}$.

\subsection{The state of play}

Over the past decade, the pursuit of optimal packings in projective spaces has mostly focused on finding ETFs.
ETFs are determined by the phase pattern in the Gram matrix, and in the real case, these phases necessarily lie in a discrete set $\{+1,-1\}$, reducing the problem to combinatorics.
As discussed earlier, real ETFs are in one-to-one correspondence with certain strongly regular graphs, and considering these graphs have been the object of intense study for over 50 years, it is likely that any news on the existence of real ETFs will not come easily.
That being said, the authors recently discovered new real ETFs in~\cite{FickusJMP:16}, and the non-existence of $19\times76$ and $20\times96$ real ETFs was recently proved with computer assistance in~\cite{AzarijaM:15,AzarijaM:16}.
Tables~1 and~3 of~\cite{FickusM:15} list the pairs $(d,n)$ that could possibly admit a real ETF, along with notes describing all existing constructions; today, $(33,66)$ and $(37,148)$ are the lowest-dimensional open cases.

In the complex setting, ETFs are still determined by the phase pattern in the Gram matrix, but this observation no longer reduces the problem to combinatorics.
Regardless, the complex case has found a lot of progress in the last decade through a multitude of approaches:
\begin{itemize}
\item
\textbf{Group actions.}
Given a finite group $G$ and a representation $\rho\colon G\rightarrow U(d)$, find a seed vector $\varphi\in\mathbb{C}^d$ such that $\{\rho(g)\varphi\}_{g\in G}$ forms an ETF after identifying collinear members of the orbit.
In the case where $G$ is abelian, the resulting ETFs correspond to difference sets in the combinatorial design literature~\cite{StrohmerH:03,XiaZG:05,DingF:07}.
The Heisenberg--Weyl group can be used to construct all known SICs~\cite{Flammia:online}, as well as an infinite family of non-SIC ETFs~\cite{IversonJM:17}.
\item
\textbf{Generalize small examples.}
Many low-dimensional ETFs can be constructed by hand with the help of some combinatorial insight; see for example Janet Tremain's influential notes on equiangular lines~\cite{Tremain:09}.
When studying such a construction, one might identify the significant combinatorial features and generalize to an infinite family of ETFs.
In this way, Tremain's notes directly led to the families in~\cite{FickusMT:12} and~\cite{FickusJMP:16}.
The hyperovals-based construction in~\cite{FickusMJ:16} followed from a similar approach.
\item
\textbf{Complexify real examples.}
Brouwer's table of strongly regular graphs~\cite{Brouwer:online} provides notes for constructing the tabulated graphs, which one might generalize to produce complex ETFs.
This approach led to the generalized quadrangle--based construction in~\cite{FickusJMPW:16}.
One may also generalize strongly regular graphs in terms of the role they play with real ETFs.
This leads to other combinatorial structures such as distance regular antipodal covers of the complete graph~\cite{CoutinhoGSZ:16} and association schemes~\cite{IversonJM:16}, both of which produce complex ETFs.
\item
\textbf{Combinatorify algebraic examples.}
In some cases, the difference sets that produce ETFs from abelian group actions can be generalized to combinatorial objects that enjoy less algebraic structure.
Under the right conditions, these more general objects produce additional ETFs, accordingly.
For example, the ETFs in~\cite{JasperMF:13} follow from studying the McFarland difference sets, while the ETFs in~\cite{FickusJ:18} are inspired by the Davis--Jedwab--Chen difference sets.
\end{itemize}
Despite all of these approaches for constructing complex ETFs, Zauner's conjecture remains elusive.
Furthermore, unlike the real case, we lack strong necessary conditions on $(d,n)$ for the existence of a complex ETF.
(The real case enjoys integrality conditions due to the integrality of phases in the Gram matrix; see for example~\cite{SustikTDH:07}.)
As such, while we now have a plethora of complex ETFs (as tabulated in~\cite{FickusM:15}), we have little concept of what remaining dimensions ought to be investigated.
In pursuit of strong necessary conditions, the first author posed the following conjecture at Sampling Theory and Applications 2015~\cite{Mixon:online}:

\begin{conjecture}
Consider the quantities
\[
d,
\qquad
n-d,
\qquad
n-1.
\]
There exists an $n$-vector equiangular tight frame in $\mathbb{C}^d$ only if one of these quantities divides the product of the other two.
\end{conjecture}

To date, the only progress has been in~\cite{Szollosi:14}, which provides a computer-assisted proof that no $8$-vector ETF exists in $\mathbb{C}^3$ or $\mathbb{C}^5$ (the result follows from 16 hours of Gr\"{o}bner basis calculation).
In these cases, the three quantities above are $3$, $5$ and $7$, and so nonexistence matches the conjecture's prediction.
Today, $(d,n)=(4,9)$ is the lowest-dimensional open case.

While the vast majority of work on projective packings has focused on achieving equality in the Welch bound, these make up a small fraction of the cases.
For the other cases, consider low-dimensional instances first:
Since $\mathbb{R}\mathbf{P}^1$ is isometrically isomorphic to the circle, optimal packings correspond to equally spaced points~\cite{BenedettoK:06}.
Similarly, $\mathbb{C}\mathbf{P}^1$ is isometrically isomorphic to the sphere, and so optimal packings correspond to spherical codes; while the optimal codes are now known for $n\leq 14$ and $n=24$ (see~\cite{MusinT:15} and references therein), the problem is open for the remaining cases (see Sloane's table~\cite{Sloane:online2} for the best known spherical codes.

The state of affairs is similar in higher dimensions:
While in the real case, integrality conditions indicate that the Welch bound is not tight for most $(d,n)$ in the Gerzon range, there is no known quantitative improvement over the Welch bound for any $(d,n)$ in this range.
Beyond the Gerzon range, alternatives like the orthoplex bound and Delsarte's linear programming bound take effect.
Equality is achieved in the orthoplex bound by mutually unbiased bases~\cite{DurtEBZ:10}, as well as various ``marriage packings''~\cite{BodmannH:16} (see Section~6.2).
Delsarte's linear programming bound can be viewed as a generalization of the Welch bound; other than ETFs and mutually unbiased bases, there are only finitely many known packings that achieve equality in this bound~\cite{HaasHM:17}.

Beyond what is described above, very little is known about optimal projective packings.
A notable exception is the case of packing $5$ points in $\mathbb{R}\mathbf{P}^2$, where the optimal packing comes from removing any vector from the $3\times 6$ ETF; this was proved by Benedetto and Kolesar in~\cite{BenedettoK:06} following the work of T\'{o}th~\cite{Toth:65} and using techniques that resemble the analysis of spherical codes~\cite{MusinT:15}.
Related work can be found in~\cite{BallingerBCGKS:09,CohnW:12,CohnHM:16}.
There has also been some work to numerically optimize packings, such as in~\cite{DhillonHST:08}.
For real projective spaces, Neil Sloane tabulates the best known packings in~\cite{Sloane:online}, which we refer to as \textbf{putatively optimal} in the sequel.
Other than the spherical codes~\cite{Sloane:online2} that correspond to packing in $\mathbb{C}\mathbf{P}^1$, there is currently no table of best known packings in complex projective spaces.
This disparity in data makes the real case far more amenable to study, as evidenced by the present paper.

\subsection{Roadmap}
\label{subsection.roadmap}

In the following section, we will present a computer-assisted proof of the optimality of a particular 6-packing in $\mathbb{R}\mathbf{P}^3$.
Our proof leverages quantifier elimination over the reals, as computed by cylindrical algebraic decomposition.
Section~3 then introduces a method of constructing packings that are within a constant factor of the Welch bound whenever $n$ lies in the Gerzon range.
This general construction relies heavily on a complex packing we introduce based on a famous character sum estimate due to Andr\'{e} Weil.
Next, we turn our attention to certifying locally optimal packings.
Our certificates are based on a general method of reformulating certain manifold optimization problems as convex programs, which we introduce in a stand-alone subsection before applying our technique to certify two infinite families of packings.
In Section~5, we introduce infinite families of near-optimal packings that arise from certain combinatorial designs, and we conclude in Section~6 by describing perfected\footnote{The putatively optimal packings in Sloane's database~\cite{Sloane:online} frequently feature emergent properties such as tightness and few angles (up to numerical precision). In these cases, we refer to the corresponding ``perfected'' version as the infinite-precision neighbor that exactly satisfies tightness and few angles, obtained with the help of cylindrical algebraic decomposition~\cite{Collins:75}, say. When perfecting these putatively optimal packings, we always check that the perfected version has strictly smaller coherence than the original.} versions of various putatively optimal packings that appear in Sloane's database~\cite{Sloane:online}.
See Table~\ref{table} for a summary of the low-dimensional instances of our results.

\section{Small optimal packings}
\label{sec.cad}

Proving the optimality of a given packing amounts to demonstrating that the packing's coherence achieves equality in a lower bound.
Along these lines, the literature contains a multitude of accomplishments involving the Welch and orthoplex bounds~\cite{FickusM:15,BoykinSTW:05,BodmannH:16}.
However, for nearly every pair $(d,n)$, the optimal $n$-packing in $\mathbb{R}\mathbf{P}^{d-1}$ does not achieve equality in either bound~\cite{FickusM:15}, indicating the need for better bounds.
In this section, we present an algorithm to compute the optimal lower bound in the case where $n=d+2$, and we apply our algorithm to solve the $d=4$ case.

First, observe that Gram matrices $G$ of $n$-packings in $\mathbb{R}\mathbf{P}^{d-1}$ with coherence at most $\mu$ form a subset of $\mathbb{R}^{n\times n}$ defined by polynomial equalities and inequalities:
\[
G^\top=G,
\quad
\operatorname{diag}(G)=\mathbf{1}, 
\quad
G\succeq0,
\quad
\operatorname{rank}(G)\leq d,
\quad
|G_{ij}|\leq\mu
\quad
\forall i,j\in[n],i\neq j.
\]
In particular, the positive semidefinite constraint may be implemented using Sylvester's criterion, specifically, by forcing the principal minors to be nonnegative.
Furthermore, the rank constraint may be implemented by asking all $(d+1)\times(d+1)$ minors to vanish.
Overall, for each Gram matrix form and sign pattern, our problem reduces to the following:
Given real polynomials $\{p_i(x_1,\ldots,x_{d+1},\mu)\}_{i\in I}$ and $\{q_j(x_1,\ldots,x_{d+1},\mu)\}_{j\in J}$, find the $\mu$'s satisfying
\begin{equation}
\label{eq.quantifier}
\exists x\in\mathbb{R}^{d+1}
\quad
\text{such that}
\quad
p_i(x,\mu)=0
\quad
\text{and}
\quad
q_j(x,\mu)\geq0
\quad
\forall i\in I,j\in J.
\end{equation}
This amounts to quantifier elimination over the reals, the plausibility of which was first demonstrated by the Tarski--Seidenberg theorem~\cite{BochnakCR:98}.
Indeed, the $\mu$'s that satisfy \eqref{eq.quantifier} are the solutions to a finite collection of univariate polynomial equalities and inequalities that are constructed by the proof of Tarski--Seidenberg.
While the implied algorithm is too slow for real-world implementation, an alternative algorithm called \textbf{cylindrical algebraic decomposition (CAD)}~\cite{Collins:75} allows for quantifier elimination over the reals with reasonable runtimes (in sufficiently small cases) and enjoys a built-in implementation in Mathematica.

As such, one could in principle use CAD to find the smallest $\mu$ for which there exists a packing of coherence at most $\mu$, but the runtime is far too slow to solve even modestly sized problems.
Instead, we will leverage combinatorics to decrease the complexity of our CAD queries.
We start with a lemma whose proof introduces some of our techniques:

\begin{lemma}
\label{lem.spanning}
For $n>d$, every optimal $n$-packing in $\mathbb{R}\mathbf{P}^{d-1}$ is a spanning set.
\end{lemma}

\begin{proof}
Given an $n$-packing $\Phi=\{\varphi_i\}_{i\in [n]}$ in $\mathbb{R}\mathbf{P}^{d-1}$, fix $\mu=\mu(\Phi)>0$ and define
\[
N(j;\Phi):=\{i\in[n]:|\langle\varphi_i,\varphi_j\rangle|=\mu\}
\]
for each $j\in[n]$.
Suppose there is no $j\in[n]$ such that $\{\varphi_i\}_{i\in N(j;\Phi)}$ forms a spanning set.
Then we may iterate through $j\in[n]$ one at a time, modifying $\varphi_j$ as follows:
Let $v_j$ denote any unit vector in the orthogonal complement of $\{\varphi_i\}_{i\in N(j;\Phi)}$, pick $t$ so that
\[
\psi_j(t)
:=\frac{\varphi_j+t v_j}{\|\varphi_j+t v_j\|}
\]
satisfies $|\langle \varphi_i,\psi_j(t)\rangle|<\mu$ for every $i\in[n]\setminus\{j\}$, and replace $\varphi_j$ with $\psi_j(t)$.
(We verify the existence of such $t$ later.)
When modifying $\Phi$ in this way, we see that $N(j;\Phi)$ becomes the empty set, whereas $j$ is removed from each $N(i;\Phi)$ with $i>j$; in particular, $\Phi$ retains the property that there is no $j\in[n]$ such that $\{\varphi_i\}_{i\in N(j;\Phi)}$ forms a spanning set, and so the iteration is well defined.
At the end of the iteration, $\Phi$ satisfies $\mu(\Phi)<\mu$, meaning the original packing was not optimal.

It remains to verify the existence of $t$.
Since $t\mapsto\psi_j(t)$ is continuous over $t\in(-1,1)$, every sufficiently small $t$ satisfies $|\langle \varphi_i,\psi_j(t)\rangle|<\mu$ for every $i\in[n]$ such that $|\langle \varphi_i,\varphi_j\rangle|<\mu$.
Meanwhile, $|\langle \varphi_i,\varphi_j\rangle|=\mu$ implies $i\in N(j;\Phi)$, and so
\[
|\langle \varphi_i,\psi_j(t)\rangle|
=\frac{|\langle \varphi_i,\varphi_j+t v_j\rangle|}{\|\varphi_j+t v_j\|}
=\frac{|\langle \varphi_i,\varphi_j\rangle|}{\|\varphi_j+t v_j\|}
\leq\frac{\mu}{\sqrt{1+t^2}}
<\mu,
\]
where the inequalities hold provided we select $t\neq0$ so that $t\langle \varphi_j,v_j\rangle\geq0$:
\[
\|\varphi_j+tv_j\|^2
=\|\varphi_j\|^2+2t\langle\varphi_j,v_j\rangle+t^2\|v_j\|^2
=1+2t\langle\varphi_j,v_j\rangle+t^2
\geq 1+t^2.
\qedhere
\]
\end{proof}

We will apply similar reasoning to identify useful combinatorial structure in optimal packings.
We require the following definition:

\begin{definition}
\label{def.secure}
\
\begin{itemize}
\item[(a)]
The \textbf{contact graph} of an $n$-packing $\Phi=\{\varphi_i\}_{i\in[n]}$ is the graph with vertex set $[n]$ and edges $\{i,j\}$ such that $|\langle\varphi_i,\varphi_j\rangle|=\mu(\Phi)$.
\item[(b)]
An $n$-vertex graph $G$ is $d$-\textbf{secure} if for every ordering of the vertices $\{v_i\}_{i\in[n]}$, there exists $j\in[n]$ such that the degree of $v_j$ in $G-\{v_i\}_{i\in[j-1]}$ is at least $d$.
\end{itemize}
\end{definition}

Here, $G-\{v_i\}_{i\in[j-1]}$ denotes the graph with vertices indexed by $[n]\setminus[j-1]$ obtained by removing from $G$ the vertices indexed by $[j-1]$, along with any incident edges.
As an example, $2$-secure graphs are precisely the graphs which contain a cycle.
See Figure~\ref{figure.secure} for an illustration.

\begin{figure}
\begin{center}
\begin{tabular}{ll}
\begin{tikzpicture}[scale=1.7]
\coordinate (1) at (2,3);
\coordinate (2) at (1,2);
\coordinate (3) at (2,2);
\coordinate (4) at (3,2);
\coordinate (5) at (4,2);
\coordinate (6) at (2,1);
\coordinate (7) at (3,1);
\draw [line width=1pt] (2) -- (3) -- (4) -- (5);
\draw [line width=1pt] (1) -- (3) -- (6);
\draw [line width=1pt] (4) -- (7);
\draw [fill=blue!20!white,thick] (1) circle [radius=0.2];
\draw [fill=blue!20!white,thick] (2) circle [radius=0.2];
\draw [fill=blue!20!white,thick] (3) circle [radius=0.2];
\draw [fill=blue!20!white,thick] (4) circle [radius=0.2];
\draw [fill=blue!20!white,thick] (5) circle [radius=0.2];
\draw [fill=blue!20!white,thick] (6) circle [radius=0.2];
\draw [fill=blue!20!white,thick] (7) circle [radius=0.2];
\node at (1){\footnotesize{$a$}};
\node at (2){\footnotesize{$b$}};
\node at (3){\footnotesize{$c$}};
\node at (4){\footnotesize{$d$}};
\node at (5){\footnotesize{$e$}};
\node at (6){\footnotesize{$f$}};
\node at (7){\footnotesize{$g$}};
\end{tikzpicture}
\hspace{0.3in}
&
\hspace{0.3in}
\begin{tikzpicture}[scale=1.7]
\coordinate (1) at (2,3);
\coordinate (2) at (1,2);
\coordinate (3) at (2,2);
\coordinate (4) at (3,2);
\coordinate (5) at (4,2);
\coordinate (6) at (2,1);
\coordinate (7) at (3,1);
\draw [line width=1pt] (2) -- (3) -- (4) -- (5);
\draw [line width=1pt] (1) -- (3) -- (6) -- (7) -- (4);
\draw [fill=blue!20!white,thick] (1) circle [radius=0.2];
\draw [fill=blue!20!white,thick] (2) circle [radius=0.2];
\draw [fill=blue!20!white,thick] (3) circle [radius=0.2];
\draw [fill=blue!20!white,thick] (4) circle [radius=0.2];
\draw [fill=blue!20!white,thick] (5) circle [radius=0.2];
\draw [fill=blue!20!white,thick] (6) circle [radius=0.2];
\draw [fill=blue!20!white,thick] (7) circle [radius=0.2];
\node at (1){\footnotesize{$a$}};
\node at (2){\footnotesize{$b$}};
\node at (3){\footnotesize{$c$}};
\node at (4){\footnotesize{$d$}};
\node at (5){\footnotesize{$e$}};
\node at (6){\footnotesize{$f$}};
\node at (7){\footnotesize{$g$}};
\end{tikzpicture}
\end{tabular}
\end{center}
\caption{
\label{figure.secure}
{\small 
Illustration of Definition~\ref{def.secure}.
For the graph on the left, it is possible to iteratively delete degree-at-most-$1$ vertices one at a time to produce the empty graph.
For example, after deleting vertices $a$, $b$ and $f$, the vertex $c$ has degree $1$ in the remaining graph, and so we may remove it before removing $g$, $d$ and $e$ (in that order).
As such, we say the graph on the left is not $2$-secure.
By contrast, the graph on the right is $2$-secure: While we can remove degree-$1$ vertices $a$, $b$ and $e$, none of the vertices in the remaining $4$-cycle have degree strictly smaller than $2$.
In general, a graph is not $d$-secure if and only if the following holds:
When iteratively deleting vertices of minimum degree, the minimum degree of the remaining graph is always strictly smaller than $d$.
}\normalsize}
\end{figure}
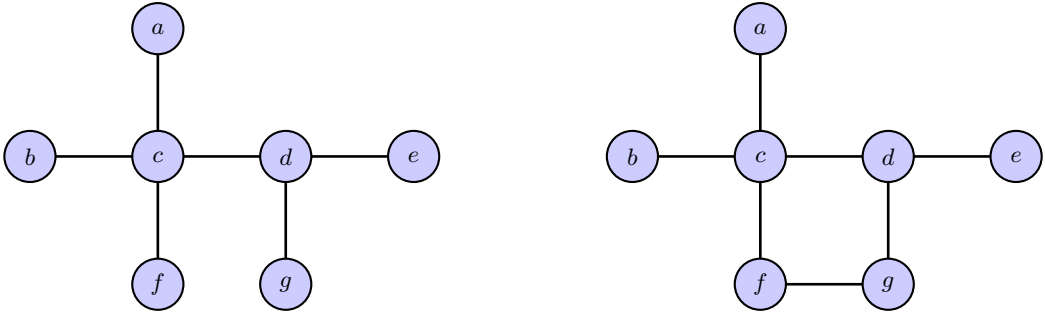

\begin{lemma}
\label{lem.secure}
For $n>d$, the contact graph of an optimal $n$-packing $\Phi$ in $\mathbb{R}\mathbf{P}^{d-1}$ is $d$-secure.
\end{lemma}

\begin{proof}
We prove the contrapositive.
Fix $\mu=\mu(\Phi)$, and given any $n$-packing $\Psi=\{\psi_i\}_{i\in[n]}$, let $G(\Psi)$ denote the graph with vertex set $[n]$ and edges $\{i,j\}$ such that $|\langle\psi_i,\psi_j\rangle|=\mu$.
Then $G(\Phi)$ is the contact graph $G$ of $\Phi=\{\varphi_k\}_{k\in[n]}$.
Suppose $G$ is not $d$-secure.
Then there exists an ordering $\{k_i\}_{i\in[n]}$ such that the degree of each $k_j$ in $G-\{k_i\}_{i\in[j-1]}$ is strictly less than $d$.
In particular, $|\langle\varphi_{k_1},\varphi_j\rangle|=\mu$ for at most $d-1$ choices of $j$.
Move $\varphi_{k_1}$ slightly into the orthogonal complement of these $\varphi_j$'s as in the proof of Lemma~\ref{lem.spanning} to produce a new $n$-packing $\Phi_1$.
Then $\mu(\Phi_1)\leq\mu(\Phi)$ and since $\mu(\Phi)>0$, we have $G(\Phi_1)=G-k_1$.
Proceeding iteratively produces $n$-packings $\Phi_2,\ldots,\Phi_{n-1}$ such that $\mu(\Phi_{n-1})\leq\cdots\leq\mu(\Phi_1)\leq\mu(\Phi)$.
Furthermore, $G(\Phi_{n-1})=G-\{k_i\}_{i\in[n-1]}$ has no edges, and so $\mu(\Phi_{n-1})<\mu(\Phi)$.
As such, $\Phi$ is not an optimal $n$-packing in $\mathbb{R}\mathbf{P}^{d-1}$.
\end{proof}

Lemma~\ref{lem.secure} is particularly telling when $n$ is small relative to $d$.
Since optimal $n$-packings in $\mathbb{R}\mathbf{P}^{d-1}$ are well understood for $n\leq d+1$, we focus on the case where $n=d+2$.
While this case was already solved for $d=2,3$ by Benedetto and Kolesar in~\cite{BenedettoK:06}, we leverage real algebraic geometry to devise a unified proof technique that solves the $d=4$ case as well.
For $d=2$, the optimal packing is the union of identity and Hadamard bases, whose coherence is $1/\sqrt{2}$.
For $d=3$, the optimal coherence is $1/\sqrt{5}$.
Here, one first selects $6$ antipodal representatives from the $12$-vertex iscosahedron to obtain an optimal $6$-packing in $\mathbb{R}\mathbf{P}^2$ before removing any vector.
Judging by Sloane's database of packings in Grassmannian spaces~\cite{Sloane:online}, this phenomenon of optimal packings arising from larger optimal packings appears to occur frequently, and we will study this further in Section~\ref{sec.local}.
For $d=4$, Sloane's database suggests that one of the optimal packings $\Phi$ satisfies
\begin{equation}
\label{eq.opt 6 in 4}
\Phi^\top\Phi=
\left[
\begin{array}{rrrrrr}
1&\mu&\mu&-\mu&-\mu&\phantom{-}\mu\\
\mu&1&-\mu&-\mu&\mu&\mu\\
\mu&-\mu&1&-\mu&\mu&\mu\\
-\mu&-\mu&-\mu&1&-\mu&\mu\\
-\mu&\mu&\mu&-\mu&1&\mu\\
\mu&\mu&\mu&\mu&\mu&1
\end{array}
\right],
\qquad
\mu=\frac{1}{3}.
\end{equation}
This packing is equiangular, but unlike the $d=3$ case, it does not appear to be obtained from a larger optimal packing.

We will prove the optimality of \eqref{eq.opt 6 in 4} by reducing to a handful of subproblems that we can solve with the help of a computer algebra system.
First, Lemma~\ref{lem.secure} forces many of the Gram matrix entries to be $\pm\mu(\Phi)$.
To see which entries necessarily have this form, we first identify the \textbf{minimal $d$-secure graphs} of order $d+2$, that is, the $d$-secure graphs of order $d+2$ with the property that no proper subgraph is $d$-secure.
Since every $d$-secure graph contains a minimal $d$-secure graph as a subgraph, this will establish which Gram matrix entries are forced.
To this end, we first identify some distinguishing properties of minimal $d$-secure graphs:

\begin{lemma}
\label{lem.minimal properties}
Every minimal $d$-secure graph contains exactly one nontrivial component.
In this component, the minimum degree is at least $d$.
Any other component amounts to an isolated vertex.
\end{lemma}

\begin{proof}
Any $d$-secure graph with multiple nontrivial components is not minimal, since one may remove the edges from one component to get a proper $d$-secure subgraph.
Given a $d$-secure graph with minimum nonzero degree less than $d$, one may remove the edges incident to the minimum-degree vertex to obtain a proper $d$-secure subgraph.
\end{proof}

\begin{lemma}
\label{lem.two minimals}
Fix $d\geq2$.
There are two minimal $d$-secure graphs of order $d+2$:
the complete graph of order $d+1$ union an isolated vertex, and the graph complement of a maximum matching.
\end{lemma}

Here, an isolated vertex is a vertex of degree zero (i.e., it is incident to zero edges, and therefore adjacent to zero vertices), whereas a maximum matching is a graph of $n$ vertices and $\lfloor n/2\rfloor$ edges such that all but possibly one vertex is adjacent to (or ``matched with'') exactly one other vertex.

\begin{proof}[Proof of Lemma~\ref{lem.two minimals}]
Lemma~\ref{lem.minimal properties} implies that every minimal $d$-secure graph $G$ must contain a component of order $\geq d+1$.
If $G$ contains a component of order $d+1$, then it must be complete in order to have minimum degree $\geq d$, and the remaining vertex must be isolated by Lemma~\ref{lem.minimal properties}.
Otherwise, $G$ is connected.
In this case, the complement of $G$ is necessarily a matching, since otherwise two edges in the complement would share a vertex, forcing that vertex to have degree less than $d$ in $G$.
Since the complement of a maximum matching is $d$-secure, this gives the only other minimal $d$-secure graph of order $d+2$.
\end{proof}

Lemma~\ref{lem.two minimals} offers substantial information about the Gram matrix of every optimal $(d+2)$-packing $\Phi$ in $\mathbb{R}\mathbf{P}^{d-1}$.
For example, when $d=2$, we may permute the columns of $\Phi$ so that
\begin{equation}
\label{eq.gram forms}
\Phi^\top\Phi=
\left[
\begin{array}{rrrr}
1&\pm\mu&\pm\mu&\phantom{\pm}x_1\\
\pm\mu&1&\pm\mu&x_2\\
\pm\mu&\pm\mu&1&x_3\\
x_1&x_2&x_3&1
\end{array}
\right]
~
\text{or}
~
\left[
\begin{array}{rrrr}
1&x_1&\pm\mu&\pm\mu\\
x_1&1&\pm\mu&\pm\mu\\
\pm\mu&\pm\mu&1&x_2\\
\pm\mu&\pm\mu&x_2&1
\end{array}
\right]
~
\text{for some}
~
x_1,x_2,x_3\in[-\mu,\mu].
\end{equation}
Indeed, in the first case, the contact graph contains the complete graph of order~$3$ union an isolated vertex, and in the second case, it contains the graph complement of a matching of size~$2$.
To demonstrate the optimality of the 4-packing of coherence $1/\sqrt{2}$ (whose Gram matrix exhibits the second form with $x_1=x_2=0$), it remains to prove that such Gram matrices do not exist for $\mu<1/\sqrt{2}$, regardless of the sign pattern.

At this point, we have a general proof technique for demonstrating the optimality of $(d+2)$-packings in $\mathbb{R}\mathbf{P}^{d-1}$: Apply Lemmas~\ref{lem.secure} and~\ref{lem.two minimals} to establish that the Gram matrix has one of two forms (as in \eqref{eq.gram forms}), and then for each sign pattern, run CAD to find a lower bound on $\mu$.
In practice, CAD is the runtime bottleneck, so we avoid this blackbox whenever possible.
To this end, we discuss two different speedups: (i) analyzing the first Gram matrix form without CAD, and (ii) identifying equivalent sign patterns to reduce the number of CAD queries.

For (i), consider the $(d+1)\times(d+1)$ submatrix $H$ of the Gram matrix obtained by removing the isolated vertex in the contact graph.
Then $H$ is the Gram matrix of $d+1$ equiangular vectors in $\mathbb{R}^d$.
Notice that conjugating $H$ with any signed permutation does not change whether $\mu$ satisfies \eqref{eq.quantifier}.
Writing $H=I+\mu S$, then $S$ captures the sign pattern (known as the \textbf{Seidel adjacency matrix} of $H$), and we say two Seidel adjacency matrices are \textbf{switching equivalent} if one can be obtained from the other by conjugating with a signed permutation.
For each $n\leq 10$, \cite{BussemakerMS:81} has determined the number $N(n)$ of switching equivalence classes of Seidel adjacency matrices of order $n$ (explicitly, $N(n)=2,3,7$ for $n=3,4,5$ respectively).
Representatives of these classes are easily obtained by drawing $S$ at random, and for $n\leq 5$, the minimum eigenvalue distinguishes the classes.
Since $H=I+\mu S$ is positive semidefinite and rank deficient, we have $\mu=-1/\lambda_\mathrm{min}(S)$.
Furthermore, $\mu$ satisfies \eqref{eq.quantifier} only if $\mu=\mu(\Phi)$ for some $(d+2)$-packing $\Phi$, and so it satisfies the Welch bound $\mu\geq\sqrt{2/(d(d+1))}$.
By exhausting through switching class representatives $S$ for each $d=2,3,4$, one observes that $\mu=-1/\lambda_\mathrm{min}(S)$ satisfies the Welch bound only if it also satisfies $\mu\geq\mu(\Psi)$, where $\Psi$ is the putatively optimal $(d+2)$-packing in $\mathbb{R}\mathbf{P}^{d-1}$.
As such, we need only consider Gram matrices of the second form, namely, those whose contact graphs contain the graph complement of a maximum maching.

For (ii), we extend switching equivalence to general matrices:
We say $S$ and $S'$ are switching equivalent if (a) $S_{ij}=0$ precisely when $S'_{ij}=0$, and furthermore (b) one can be obtained from the other by conjugating with a signed permutation.
Let $G$ be a Gram matrix of the second form and write $G=B+\mu S$, where $B$ is a block-diagonal matrix with $1$'s on the diagonal and $x_i$'s and $0$'s on the off-diagonal, and all the entries of $S$ lie in $\{0,\pm1\}$.
As before, replacing $S$ with a switching equivalent $S'$ does not change whether $\mu$ satisfies \eqref{eq.quantifier}, and so it suffices to restrict our CAD queries to switching class representatives.
In the case where $d=4$, the number of representatives is 14; see~\cite{FickusJM:17a} for a Mathematica script that iterates through the corresponding CAD queries in about 30 seconds.
This gives the main result of this section:

\begin{theorem}
\label{thm.6in4}
The $6$-packings in $\mathbb{R}\mathbf{P}^3$ that satisty \eqref{eq.opt 6 in 4} are optimal.
\end{theorem}

Since the published Benedetto--Kolesar proof of the $d=3$ case omits certain details for the sake of presentation, we also provide a Mathematica-assisted proof of this case in~\cite{FickusJM:17b}.
We suspect that our methods generalize to larger $d$, but doing so apparently requires either additional computational resources or clever quantifier elimination (e.g., permuting the variables, relaxing the polynomial constraints, or applying a specialized alternative to CAD).
In anticipation of these developments, we offer the following analytic Gram matrices for the $d=5,6$ cases from Sloane's database~\cite{Sloane:online}:
\begin{equation}
\label{eq.conjectured grams}
G_5=
\left[
\footnotesize{
\begin{array}{rrrrrrr}
     1  &  -a   &  a  &  -a  &   \phantom{-}a  &  -a  &   a\\
    -a  &   1   &  a  &   a  &   a  &  -a  &   a\\
     a  &   a   &  1  &  -a  &   a  &   a  &  -a\\
    -a  &   a   & -a  &   1  &   a  &  -a  &  -a\\
     a  &   a   &  a  &   a  &   1  &   a   &  a\\
    -a  &  -a   &  a  &  -a  &   a &    1  &  -a\\
     a  &   a   & -a  &  -a  &   a &   -a  &   1
     \end{array}
     }
     \right],
\quad
     G_6=
     \left[
     \footnotesize{
\begin{array}{rrrrrrrr}
    1  &   b  &   b  &  -b  &   b  &   c  &   b  &  -b\\
     b  &  1  &  -b  &  -b  &  -b  &  -b  &  -c  &  -b\\
     b  &  -b  &  1  &  -b  &  -b  &  -b  &  -b  &  -b\\
    -b  &  -b  &  -b  &  1  &   b  &  -b  &   b  &  -b\\
     b  &  -b  &  -b  &   b  &  1 &   -b  &  -b  &   b\\
     c  &  -b  &  -b  &  -b  &  -b  &  1  &   b  &  -b\\
     b  &  -c  &  -b  &   b  &  -b  &   b  &  1  &   b\\
    -b  &  -b  &  -b  &  -b  &   b  &  -b  &   b  &  1\\
         \end{array}
         }
     \right],
\end{equation}
where $a>0$ is the second smallest root of $x^3 - 9x^2 - x + 1$, $b>0$ is the second smallest root of 
\begin{equation}
\label{eq.6by8coherence}
106x^6 - 264x^5 - 53x^4 + 84x^3 + 20x^2 - 4x - 1,
\end{equation}
and $c\in(0,b)$ is the fourth smallest root of 
\[
53x^6+484x^5+814x^4-860x^3-347x^2+352x-32.
\]
Indeed, the Gram matrices in \eqref{eq.conjectured grams} match the form of Sloane's numerical constructions, and so the exact value for $a$ is $-1/\lambda_\mathrm{min}(S)$, where $S$ is the corresponding Seidel adjacency matrix, whereas $b$ and $c$ can be obtained by passing to CAD.
Due to the simplicity of the former case, Table~\ref{table} provides perfected versions of all equiangular putatively optimal packings in Sloane's database~\cite{Sloane:online}, excluding the equiangular packings that are subpackings of larger equiangular packings.
We note that to date, the squared coherence of every known optimal packing in real projective space is rational.
In light of this, the above putatively optimal packings are striking---the coherence of $G_6$ is not even expressible by radicals!
This motivates the following guarantee on the field structure of optimal coherence:

\begin{theorem}
The coherence of an optimal packing in real projective space is algebraic.
\end{theorem}

\begin{proof}
Let $S$ denote the semialgebraic set of Gram matrices of $n$-packings in $\mathbb{R}\mathbf{P}^{d-1}$:
\[
S
=\Big\{G\in\mathbb{R}^{n\times n}:G^\top=G,\operatorname{diag}(G)=\mathbf{1}, 
G\succeq0,
\operatorname{rank}(G)\leq d
\Big\}.
\]
Let $\Phi$ be an optimal $n$-packing in $\mathbb{R}\mathbf{P}^{d-1}$, and consider the set
\[
T=\Big\{(G,x):G\in S,-x\leq G_{ij}\leq x ~ \forall i,j\in[n],i\neq j\Big\}.
\]
Then $\mu(\Phi)=\min_{(G,x)\in T}x$.
Since $T$ is a semialgebraic set defined by polynomials with rational coefficients, then the Tarski--Seidenberg theorem (specifically, Theorem~1.4.2 in~\cite{BochnakCR:98}) gives that the set $\operatorname{proj}_xT\subseteq\mathbb{R}$ of all $x$ for which there exists $G\in S$ such that $(G,x)\in T$ is also a semialgebraic set defined by polynomials with rational coefficients.
As such, $\mu(\Phi)=\min(\operatorname{proj}_xT)$ is algebraic.
\end{proof}

\section{Approximately optimal packings}

At this point, the reader may appreciate the difficulty involved in constructing provably optimal packings in real projective space.
In this section, we offer a recipe to construct packings whose coherence is within a constant factor of optimal provided $n$ lies in the Gerzon range~\eqref{eq.gerzon}.
We begin with a versatile packing in complex projective space:

\begin{theorem}
\label{thm.weil}
Let $\psi$ be a nontrivial additive character of $\mathbb{F}_q$.
For each $f\in\mathbb{F}_q[x]$, define
\[
\varphi_f(x)
=\frac{1}{\sqrt{q}}\psi(f(x))
\qquad
\forall x\in\mathbb{F}_q.
\]
For each $r<\operatorname{char}(\mathbb{F}_q)$, let $S(r)$ denote the $f$'s such that $f(0)=0$ and $\mathrm{deg}(f)\leq r$.
Then
\begin{itemize}
\item[(a)]
$S(1)\subseteq S(2)\subseteq\cdots\subseteq S(\operatorname{char}(\mathbb{F}_q)-1)$ with each $\{\varphi_f\}_{f\in S(r)}$ formed by $q^r$ vectors in $\mathbb{C}^q$.
\item[(b)]
$\{\varphi_f\}_{f\in S(1)}$ is formed by the additive characters of $\mathbb{F}_q$.
\item[(c)]
$\{\varphi_f\}_{f\in S(2)}$ gives $q$ bases in $\mathbb{C}^q$ that, together with the identity basis, are mutually unbiased.
\item[(d)]
$|\langle \varphi_f,\varphi_g\rangle|\leq(r-1)/\sqrt{q}$ for all $f,g\in S(r)$ with $f\neq g$.
\end{itemize}
\end{theorem}

Parts (a) and (b) are straightforward.
Part (c) is well known; see~\cite{CasazzaF:06} for example.
Part (d) follows immediately from a celebrated result of Andr\'{e} Weil, specifically, Theorem 2E in~\cite{Schmidt:76}.
We will use this construction to form near-optimal packings in real projective space with the help of two operations: the $\mathbb{C}$-to-$\mathbb{R}$ trick and the Naimark complement.
To be clear, the $\mathbb{C}$-to-$\mathbb{R}$ trick refers to the replacement
\[
a+ib
\mapsto
\left[
\begin{array}{rr}
a&-b\\
b&a
\end{array}
\right].
\]
This operation converts $m\times n$ complex matrices into $2m\times2n$ real matrices with similar properties.
For example, applying the $\mathbb{C}$-to-$\mathbb{R}$ trick to a complex packing (such as the ones in Theorem~\ref{thm.weil}) produces a real packing of smaller or equal coherence.

\begin{theorem}
\label{thm.approx}
If $d+\sqrt{2d+1/4}+1/2\leq n\leq d(d+1)/2$, then there exists an $n$-packing $\Phi$ in $\mathbb{R}\mathbf{P}^{d-1}$ such that
\[
\mu(\Phi)
\leq20\sqrt{6}\cdot\sqrt{\frac{n-d}{d(n-1)}}.
\]
\end{theorem}

\begin{proof}
The upper bound is nontrivial (i.e., less than $1$) for some $n$ in the Gerzon range only if $d\geq217$, and so we may assume $d\geq217$ without loss of generality.
We consider three cases:
\medskip

\noindent
\textbf{Case I:} $(5/4) d\leq n \leq d(d+1)/2 $.
Let $p$ be the largest prime satisfying $2p \leq d$, and construct $\{\varphi_f\}_{f\in S(3)}$ with $q=p$.
Apply the $\mathbb{C}$-to-$\mathbb{R}$ trick to produce $2p^3$ unit vectors in $\mathbb{R}^{2p}$ with coherence at most $2/\sqrt{p}$.
Select the first $n$ of these vectors and embed in $\mathbb{R}^d$ to produce an $n$-packing $\Phi$ with 
\[
\mu(\Phi)
\leq\frac{2}{\sqrt{p}}
\leq\frac{4}{\sqrt{d}}
\leq 4\sqrt{5}\cdot\sqrt{\frac{n-d}{d(n-1)}},
\]
where the second step applies Bertrand's postulate and the last step uses $n\geq(5/4)d$.
The fact that $n\leq d(d+1)/2\leq 2p^3$ follows from Bertrand's postulate and $p\geq5$.
\medskip

\noindent
\textbf{Case II:} $d+\sqrt{2d+1}+1\leq n <(5/4)d$.
Let $p$ the be smallest prime satisfying $2p\geq n-d$.
Put $k=\lceil n/(2p)\rceil$ and let $\{U_j\}_{j=0}^{k-1}$ denote mutually unbiased bases over $\mathbb{C}^p$.
The fact that $k\leq p+1$ follows from the assumed lower bound on $n$, which can be rearranged to say $2n\leq(n-d)^2$; indeed, this gives
\[
k
=\left\lceil \frac{n}{2p}\right\rceil
\leq\left\lceil \frac{n}{n-d}\right\rceil
\leq \frac{n}{n-d}+1
\leq \frac{n-d}{2}+1
\leq p+1.
\]
Let $\omega$ denote a primitive $k$th root of unity, and consider the $2p\times kp$ matrix $A$ whose $(a,b)$th $p\times p$ submatrix is given by $(1/\sqrt{2})\omega^{ab}U_b$.
Then the columns of $A$ form a unit norm tight frame with coherence at most $1/\sqrt{p}$.
Apply the $\mathbb{C}$-to-$\mathbb{R}$ trick to produce a tight frame of $2kp$ unit vectors in $\mathbb{R}^{4p}$ with coherence at most $1/\sqrt{p}$.
Taking the Naimark complement then gives a tight frame of $2kp$ unit vectors in $\mathbb{R}^{2(k-2)p}$ with coherence at most $2/((k-2)\sqrt{p})$.
Select the first $n$ of these vectors and embed in $\mathbb{R}^d$ to produce an $n$-packing $\Phi$ with
\[
\mu(\Phi)
\leq\frac{2}{(k-2)\sqrt{p}}
\leq\frac{2}{(\frac{n}{2p}-2)\sqrt{p}}
\leq\frac{2}{(\frac{n}{2(n-d)}-2)\sqrt{\frac{n-d}{2}}}
=4\sqrt{2}\cdot\frac{\sqrt{n-d}}{4d-3n},
\]
where the third step applies Bertrand's postulate.
Since $n<(5/4)d$, we further have
\[
4d-3n>n/5\geq\sqrt{dn}/5>\sqrt{d(n-1)}/5,
\]
with which we may continue the above estimate:
\[
\mu(\Phi)
\leq4\sqrt{2}\cdot\frac{\sqrt{n-d}}{4d-3n}
\leq 20\sqrt{2}\cdot\sqrt{\frac{n-d}{d(n-1)}}.
\]
\medskip

\noindent
\textbf{Case III:} $d+\sqrt{2d+1/4}+1/2\leq n <d+\sqrt{2d+1}+1$.
For each $d$, there is at most one value of $n$ in this case.
As such, take $n'=n+1$, apply the method of Case II, and remove the last vector to get an $n$-packing $\Phi$ in $\mathbb{R}^d$ with
\[
\mu(\Phi)
\leq20\sqrt{2}\cdot\sqrt{\frac{n'-d}{d(n'-1)}}
\leq20\sqrt{2}\cdot\sqrt{\frac{n-d+1}{n-d}}\cdot\sqrt{\frac{n-d}{d(n-1)}}.
\]
As this point, we apply our bounds on $n$:
\[
\sqrt{\frac{n-d+1}{n-d}}
\leq\sqrt{\frac{\sqrt{2d+1}+2}{\sqrt{2d+1/4}+1/2}}
\leq\sqrt{3},
\]
and combining with the previous estimate gives the result.
\end{proof}

We did not attempt to optimize the constant $20\sqrt{6}\approx 48.99$, leaving this for a possible student project.
Judging by Sloane's database~\cite{Sloane:online}, we expect the optimal constant to be less than~2.
The above proof suggests an initialization for a local optimization routine, which amounts to a 49-approximation algorithm for optimal packings in the Gerzon range.
In this spirit, the next section offers sufficient conditions for packings to be locally optimal.

\section{Locally optimal packings}
\label{sec.local}

In this section, we study packings that are locally optimal.
To do so, we first develop some manifold optimization theory that we suspect enjoys applications beyond the scope of this paper (e.g., covariance estimation~\cite{FanLL:16}).
As such, we package this more general material into the following self-contained subsection before applying it to our problem.

\subsection{Passing to convexity: An aside}

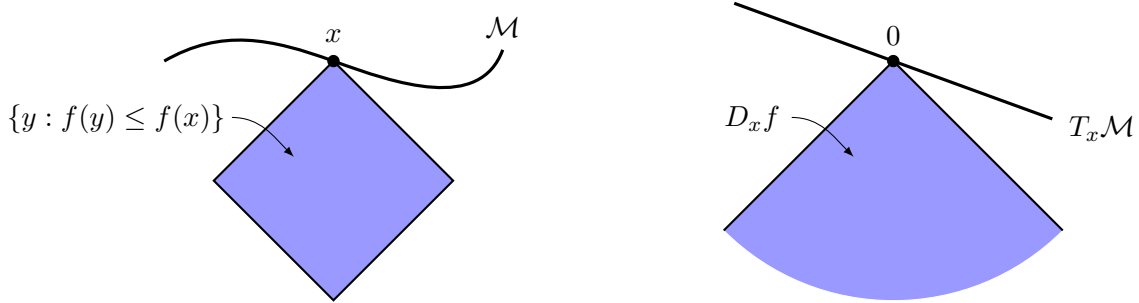
\begin{figure}
\begin{center}
\begin{tabular}{ll}
\begin{tikzpicture}[scale=0.75,>=latex]
\clip (-6,-4.5) rectangle (4,1.5);
\filldraw[thick]
[fill=blue!40!white, draw=black]
[shift={(0,0)},rotate=-135] (0,0) (0,0) rectangle (3,3);
\draw [very thick] (-3,0) to[out=30,in=160] (0,0) to[out=-20,in=250] (3,0.2);
\filldraw [black] (0,0) ellipse [x radius=0.1cm, y radius=0.1cm];
\draw (0,0.45) node {$x$};
\draw (3,0.6) node {$\mathcal{M}$};
\draw [->] (-1.5-0.3,-0.7-0.3) to[out=0,in=135] (-0.4-0.3,-1.4-0.3);
\draw (-3.55-0.3,-0.73-0.3) node {$\{y:f(y)\leq f(x)\}$};
\end{tikzpicture}
\hspace{0.5in}
&
\begin{tikzpicture}[scale=0.75,>=latex]
\clip (-3.5,-4.5) rectangle (4.5,1.5);
\filldraw[fill=blue!40!white, rotate=-135, draw=white]
(0,0) -- (4.25,0) arc (0:90:4.25) -- (0,0);
\draw [thick, rotate=-45] (0,0) -- (4.25,0);
\draw [thick, rotate=-135] (0,0) -- (4.25,0);
\draw [very thick, rotate=-20] (-3,0) -- (3,0);
\filldraw [black] (0,0) ellipse [x radius=0.1cm, y radius=0.1cm];
\draw (0,0.45) node {$0$};
\draw (3.7,-1.2) node {$T_x\mathcal{M}$};
\draw [->] (-1.5-0.3,-0.7-0.3) to[out=0,in=135] (-0.4-0.3,-1.4-0.3);
\draw (-2.2-0.3,-0.73-0.3) node {$D_xf$};
\end{tikzpicture}
\hspace{0.5in}
\end{tabular}
\end{center}
\caption{
\label{figure.local opt}
{\small 
Illustration of Theorem~\ref{thm.strong local general}.
We are interested in determining whether $x\in\mathcal{M}$ is a local minimizer of $f$.
We may locally model both $\mathcal{M}$ and the sublevel set of $x$ in terms of convex objects, namely, the tangent space $T_x\mathcal{M}$ and descent cone $D_xf$, respectively.
Suppose the tangent space intersects the descent cone uniquely at the origin, which can be certified with the help of the dual convex program.
If $f$ is polytopic, then the pointiness of the sublevel set coupled with the smoothness of $\mathcal{M}$ gives that the only nearby $z\in \mathcal{M}$ in the sublevel set of $x$ is $z=x$.
As such, $x$ is a strong local minimizer of $f$ in $\mathcal{M}$.
}\normalsize}
\end{figure}

We say a function $f\colon\mathbb{R}^m\rightarrow\mathbb{R}$ is \textbf{polytopic} if every sublevel set $\{x\in\mathbb{R}^m:f(x)\leq z\}$ is a finite intersection of closed halfspaces.
Here and throughout, given a smooth manifold $\mathcal{M}$ and a point $x\in\mathcal{M}$, we denote the tangent and normal spaces of $\mathcal{M}$ at $x$ by $T_x\mathcal{M}$ and $N_x\mathcal{M}$, respectively.
We say $x\in\mathcal{M}$ is a \textbf{strong local minimizer} of $f$ in $\mathcal{M}$ if there is a neighborhood $\mathcal{N}_x$ of $x$ such that every $y\in\mathcal{M}\cap\mathcal{N}_x\setminus\{x\}$ satisfies $f(y)>f(x)$.

\begin{theorem}
\label{thm.strong local general}
Take any polytopic function $f\colon\mathbb{R}^m\rightarrow\mathbb{R}$ and smooth manifold $\mathcal{M}\subseteq\mathbb{R}^m$.
If $x\in\mathcal{M}$ is the unique minimizer of $f(z)$ subject to $z-x\in T_x\mathcal{M}$, then $x$ is also a strong local minimizer of $f(z)$ subject to $z\in\mathcal{M}$ 
\end{theorem}

We are particularly interested in this result when $f$ is convex, e.g., $f(z)=\|z\|_1$ or $\|z\|_\infty$, since certifying unique minimizers in such cases is a well-established enterprise.
However, the proof of Theorem~\ref{thm.strong local general} does not require convexity.
Rather, as illustrated in Figure~\ref{figure.local opt}, it uses the fact that any sublevel set of a polytopic function is locally ``pointy,'' whereas a smooth manifold is locally flat:

\begin{lemma}
\label{lem.curvature}
Given a smooth manifold $\mathcal{M}\subseteq\mathbb{R}^m$, fix $x\in\mathcal{M}$. 
There exist $\epsilon,c>0$ such that every $y\in\mathcal{M}$ with $\|y-x\|\leq\epsilon$ satisfies
\[
\|\operatorname{proj}_{N_x\mathcal{M}}(y-x)\|
\leq c\|y-x\|^2.
\]
\end{lemma}

We note that Theorem~\ref{thm.strong local general} does not hold for convex functions in general.
For example, suppose $f(z_1,z_2)=\sqrt{z_1^2+z_2^2}$ and $\mathcal{M}=\{(z_1,z_2):z_1^2+2z_2^2=1\}$, and put $x=(1,0)$.
Then $x$ is the unique minimizer of $f(z)$ subject to $z-x\in T_x\mathcal{M}=\{0\}\times\mathbb{R}$, but $x$ fails to be a strong local minimizer over $\mathcal{M}$.
On the contrary, $x$ is a global maximizer over $\mathcal{M}$!

Also, the converse of Theorem~\ref{thm.strong local general} does not hold.
To see this, suppose $f(z_1,z_2)=\max\{|z_1|,|z_2|\}$ and $\mathcal{M}=\{(z_1,z_2):(z_1-2)^2+z_2^2=1\}$, and put $x=(1,0)$.
Then $x$ is a strong local minimizer over $\mathcal{M}$ (in fact, the unique global minimizer), and $x$ is also a minimizer of $f(z)$ subject to $z-x\in T_x\mathcal{M}=\{0\}\times\mathbb{R}$, but it fails to be unique.

\begin{proof}[Proof of Theorem~\ref{thm.strong local general}]
Consider the descent cone $D_xf$ generated by all $y-x$ such that $y\in\mathbb{R}^m$ with $f(y)\leq f(x)$.
Our uniqueness assumption implies $D_xf\cap T_x\mathcal{M}=\{0\}$, and so $\|\operatorname{proj}_{N_x\mathcal{M}}z\|>0$ for all $z\in D_xf\setminus\{0\}$.
Since $f$ is polytopic, $D_xf$ is a finite intersection of closed halfspaces, and so compactness gives $a>0$ such that $\|\operatorname{proj}_{N_x\mathcal{M}}z\|\geq a$ for every $z\in D_xf\cap\mathbb{S}^{m-1}$.
To prove the theorem, suppose to the contrary that there is a sequence $\{x_i\}_{i=1}^\infty$ in $\mathcal{M}$ converging to $x$ with $x_i\neq x$ and $f(x_i)\leq f(x)$ for all $i$.
Then for each $i$, we have $(x_i-x)/\|x_i-x\|\in D_xf$, and so
\[
a\|x_i-x\|
\leq \left\|\operatorname{proj}_{N_x\mathcal{M}}\frac{x_i-x}{\|x_i-x\|}\right\|\cdot\|x_i-x\|
=\|\operatorname{proj}_{N_x\mathcal{M}}(x_i-x)\|
\leq c\|x_i-x\|^2,
\]
where the last step holds for sufficiently large $i$ by Lemma~\ref{lem.curvature}.
Rearranging then gives $\|x_i-x\|\geq a/c$ for all sufficiently large $i$, contradicting the fact that $x_i\rightarrow x$.
\end{proof}

It remains to prove Lemma~\ref{lem.curvature}, which describes how flat $\mathcal{M}$ is in a neighborhood of $x$.
The proof amounts to an application of Taylor's theorem on the geodesics of $\mathcal{M}$ emanating from $x$.
The reader who is unacquainted with the ideas in the following proof is encouraged to consult a reference text in differential geometry, for example~\cite{GodinhoN:14}.

\begin{proof}[Proof of Lemma~\ref{lem.curvature}]
There is an open set $U\subseteq T_x\mathcal{M}$ containing $0_x$ such that for each $v\in U\setminus\{0\}$, there exists a geodesic $c_v\colon[0,1]\rightarrow\mathcal{M}$ with $c_v(0)=x$ and $c_v'(0)=v$.
Furthermore, the distance between $c_v(t)$ and $x$ along $\mathcal{M}$ is $\operatorname{dist}_\mathcal{M}(c_v(t),x)=\|tv\|$.
The exponential map $\operatorname{exp}_x\colon U\rightarrow M$ is defined in terms of these geodesics by $\operatorname{exp}_x(v)=c_v(1)$.
Pick $\delta>0$ such that the closed $\delta$-ball $\overline{B}_\delta$ in $T_x\mathcal{M}$ is contained in $U$.
Since the exponential map is a diffeomorphism onto some open subset of $\mathcal{M}$ containing $x$, compactness implies
\[
C=\max_{v\in\partial\overline{B}_\delta}\max_{t\in[0,1]}|c_v''(t)|
\]
is well defined.
For every $y\in\operatorname{exp}_x(\overline{B}_\delta)$, there exists $v\in\partial\overline{B}_\delta$ and $t\in[0,1]$ such that $y=c_v(t)$, and so Taylor's theorem gives
\[
\|y-(x+tv)\|
=\|c_v(t)-(c_v(0)+tc_v'(0))\|
\leq \frac{1}{2}Ct^2
= \frac{C}{2\delta^2}\Big(\operatorname{dist}_\mathcal{M}(y,x)\Big)^2.
\]
The projection theorem then gives
\[
\|\operatorname{proj}_{N_x\mathcal{M}}(y-x)\|
=\min_{z\in T_x\mathcal{M}}\|(y-x)-z\|
\leq\|y-(x+tv)\|
\leq\frac{C}{2\delta^2}\Big(\operatorname{dist}_\mathcal{M}(y,x)\Big)^2.
\]
To conclude the proof, we will show that $\operatorname{dist}_\mathcal{M}(y,x)\leq 2\|y-x\|$ whenever
\[
\operatorname{dist}_\mathcal{M}(y,x)
\leq \eta:=\min\left\{\delta,\frac{\delta^2}{2C}\right\}.
\]
Indeed, this will prove the lemma with $c=2C/\delta^2$ and taking $\epsilon$ to be the smallest $\|y-x\|$ such that $\operatorname{dist}_\mathcal{M}(y,x)=\eta$, which exists by compactness.
To this end, we have $\operatorname{dist}_\mathcal{M}(y,x)\leq\delta$, and so $y\in\operatorname{exp}_x(\overline{B}_\delta)$.
As such, we may pick $c_v$ as before to get
\begin{align*}
\operatorname{dist}_\mathcal{M}(y,x)
&=\int_0^t\|c_v'(s)\|ds\\
&\leq\int_0^t\left(\frac{1}{t}\|y-x\|+\frac{1}{t}\|y-(x+tv)\|+\|c_v'(s)-v\|\right)ds\\
&\leq \|y-x\|+\frac{1}{2}Ct^2+\int_0^t\|c_v'(s)-c_v'(0)\|ds\\
&\leq\|y-x\|+Ct^2\\
&=\|y-x\|+\frac{C}{\delta^2}\Big(\operatorname{dist}_\mathcal{M}(y,x)\Big)^2\\
&\leq\|y-x\|+\frac{1}{2}\operatorname{dist}_\mathcal{M}(y,x),
\end{align*}
where the second and third inequalities apply Taylor's theorem and the last step follows from the fact that $\operatorname{dist}_\mathcal{M}(y,x)\leq\delta^2/(2C)$.
Rearranging then gives the result.
\end{proof}

\subsection{Certifying strongly locally optimal packings}

Let $\mathcal{M}_d^n$ denote the set of Gram matrices of spanning $n$-packings in $\mathbb{R}\mathbf{P}^{d-1}$:
\begin{equation}
\label{eq.gram manifold}
\mathcal{M}_d^n
=\{G\in\mathbb{R}^{n\times n}:G^\top =G,~\operatorname{diag}(G)=\mathbf{1},~G\succeq0,~\operatorname{rank}(G)=d\}
\end{equation}
One may verify that $\mathcal{M}_d^n$ is an embedded submanifold of $\mathbb{R}^{n\times n}$ using standard techniques discussed in Section~3.3 of~\cite{AbsilMS:08}.
Since optimal line packings necessarily span (Lemma~\ref{lem.spanning}), we may restrict our optimization to this manifold:
\begin{equation}
\label{eq.manifold optimization}
\tag{MP}
\text{minimize }\|X-I\|_\infty\text{ subject to }X\in\mathcal{M}_d^n
\end{equation}
where $\|\cdot\|_\infty$ denotes the entrywise $\infty$-norm.
We say an $n$-packing in $\mathbb{R}\mathbf{P}^{d-1}$ is \textbf{strongly locally optimal} if its Gram matrix is a strong local minimizer of \eqref{eq.manifold optimization}.
Suppose we want to certify that a given $G\in\mathcal{M}_d^n$ is such a strong local minimizer.
Each sublevel set $\{X\in\mathbb{R}^{n\times n}:\|X-I\|_\infty\leq z\}$ is the intersection of closed halfspaces $H_{ijk}=\{X\in\mathbb{R}^{n\times n}:(-1)^k(X_{ij}-\delta_{ij})\leq z\}$.
As such, $X\mapsto\|X-I\|_\infty$ is polytopic, and so we may apply Theorem~\ref{thm.strong local general} to pass to a convex program:
\begin{equation}
\label{eq.primal}
\tag{P$_G$}
\text{minimize }\|X-I\|_\infty\text{ subject to }X-G\in T_G\mathcal{M}_d^n
\end{equation}
We will certify that $G$ uniquely minimizes \eqref{eq.primal}, thereby certifying that $G$ is a strong local minimizer of \eqref{eq.manifold optimization}, with the help of a dual certificate.
To this end, the dual program is given by
\begin{equation}
\label{eq.dual}
\tag{D$_G$}
\text{maximize }\operatorname{tr}((G-I)Y)\text{ subject to }Y\in N_G\mathcal{M}_d^n,~\|Y\|_1\leq1
\end{equation}
where $\|\cdot\|_1$ denotes the entrywise 1-norm; see the appendix for a derivation.

\begin{lemma}
\label{lem.dual cert}
Take $n>d$, define $\mathcal{M}_d^n$ by \eqref{eq.gram manifold}, and fix $G\in\mathcal{M}_d^n$.
Then $G$ minimizes \eqref{eq.primal} if and only if there exists $0\neq Y\in N_G\mathcal{M}_d^n$ with
\begin{itemize}
\item[(a)]
$(G_{ij}-\delta_{ij})Y_{ij}\geq0$ for all $i,j\in[n]$, and
\item[(b)]
$Y_{ij}=0$ for all $i,j\in[n]$ such that $|G_{ij}-\delta_{ij}|<\|G-I\|_\infty$.
\end{itemize}
In this case, put $S=\{(i,j)\in[n]^2:Y_{ij}\neq0\}$ and define the linear operator $L\colon Z\mapsto\{Z_{ij}\}_{(i,j)\in S}$.
Then $G$ is the unique minimizer of \eqref{eq.primal} if $L$ restricted to $T_G\mathcal{M}_d^n$ is injective.
\end{lemma}

\begin{proof}
First, we note that (a) and (b) are together equivalent to $\operatorname{tr}((G-I)Y)=\|G-I\|_\infty\|Y\|_1$.
For ($\Leftarrow$), we normalize $\hat{Y}=Y/\|Y\|_1$ for dual feasibility, and then every primal feasible $X$ satisfies
\[
\|G-I\|_\infty
=\operatorname{tr}((G-I)\hat{Y})
\leq\|X-I\|_\infty
\]
by weak duality.
For ($\Rightarrow$), let $Y$ denote any maximizer of \eqref{eq.dual}.
Then strong duality (which holds trivially by Slater's condition~\cite{BoydV:04}) implies $\operatorname{tr}((G-I)Y)=\|G-I\|_\infty\|Y\|_1$, as desired.
Furthermore, the fact that $Y\neq0$ follows from Welch's lower bound on the value of \eqref{eq.manifold optimization}:
\[
0
<\sqrt{\frac{n-d}{d(n-1)}}
\leq\operatorname{val}\eqref{eq.manifold optimization}
\leq\|G-I\|_\infty
=\operatorname{tr}((G-I)Y),
\]
where the last step again applies strong duality.

Next, we note that by strong duality, every minimizer $X$ of \eqref{eq.primal} necessarily satisfies (a) and (b) with $G$ replaced by $X$.
In particular, every such $X$ must be identical to $G$ at the entries indexed by $S$.
As such, $L(X-G)=0$, and so injectivity implies $X=G$.
\end{proof}

The remainder of this section uses Lemma~\ref{lem.dual cert} to prove the strong local optimality of two infinite families of packings.

\begin{corollary}
\label{cor.srg}
Let $A$ and $B$ denote the adjacency matrices of a $(v,k)$-strongly regular graph and its complement, and let $-\beta<0$ and $m$ denote the smallest eigenvalue of $A$ and its multiplicity, respectively.
Suppose $A$ satisfies $2k+1\neq v<2(k+\beta)$.
Putting $d=v-m$ and $\mu=1/(2\beta-1)$, then $I+\mu A-\mu B$ is the Gram matrix of a strongly locally optimal $v$-packing in $\mathbb{R}\mathbf{P}^{d-1}$.
\end{corollary}

One can show that whenever a $(v,k)$-strongly regular graph with $v\neq 2k+1$ exists, either the graph or its complement satisfies $v<2(k+\beta)$.
In many cases, the strongly regular graph corresponds to an ETF \`{a} la~\cite{Waldron:09}, and so removing a vector from that ETF produces a strongly locally optimal packing.
In fact, this is how the optimal $5$-packing in $\mathbb{R}\mathbf{P}^2$ is constructed, and Sloane's database suggests that for many $d\times n$ ETFs with $n\geq2d$, one can get away with removing additional vectors.
However, when $n<2d$, removing a vector from an ETF may be suboptimal, for example, Sloane's database identifies a biangular 15-packing in $\mathbb{R}\mathbf{P}^9$ that exhibits lower coherence than the equiangular packing that corresponds to the $(15,6,1,3)$-strongly regular graph, which in turn is obtained by removing a vector from the $10\times16$ ETF.

The $(40,27,18,18)$-strongly regular graph produces a $40$-packing in $\mathbb{R}\mathbf{P}^{15}$ that does not come from an ETF.
Interestingly, this example improves on the corresponding packing in Sloane's database by over 6 percent, and removing up to 5 vectors from this packing also produces improvements. 
Similarly, the $(36,21,12,12)$-strongly regular graph produces a $15\times36$ ETF, and improvements arise by removing up to 5 vectors from this packing as well.

Our proof uses the following result on eigenvalue integrality:

\begin{proposition}[see Lemma~8 in~\cite{SustikTDH:07}]
\label{prop.int evals}
Let $S$ be a symmetric matrix with integer entries whose eigenvalues have distinct multiplicities.
Then every eigenvalue of $S$ is integer.
\end{proposition}

\begin{proof}[Proof of Corollary~\ref{cor.srg}]
The eigenvalues of $A$ are $k$, $\alpha$ and $-\beta$ for some $\alpha,\beta>0$.
Since $v\neq 2k+1$, Proposition~\ref{prop.int evals} gives that $\alpha$ and $\beta$ are integer, and furthermore, the Perron--Frobenius theorem gives that $k>\max\{\alpha,\beta\}$.
Let $\hat{J}$ denote the orthogonal projection onto the all-ones vector, and let $P$ and $Q$ denote orthogonal projections onto the other eigenspaces of $A$ so that
\[
A
=k\hat{J}+\alpha P-\beta Q.
\]
Put $G=I+\mu A-\mu B$, and observe the integrality of $\beta>0$ implies $\mu>0$.
Considering $B=v\hat{J}-I-A$, we then have
\[
(2\beta-1)G
=(2\beta-1)I+A-B
=\big(2(k+\beta)-v\big)\hat{J}+2(\alpha+\beta)P.
\]
Since $v<2(k+\beta)$ by assumption, $G$ is positive semidefinite with rank $d$.

To demonstrate strong local optimality of the corresponding packing, Theorem~\ref{thm.strong local general} gives that it suffices to show $G$ is the unique global minimizer of \eqref{eq.primal}.
To this end, we will select $Y$ to be $\hat{J}+P$ less its diagonal component $D$ and apply Lemma~\ref{lem.dual cert}.
Indeed, $Y\neq0$ because $\hat{J}$ and $P$ lie in the span of $\{I,A,B\}$, and so $D=cI$ for some $c\in\mathbb{R}$, implying $\operatorname{rank} D\neq\operatorname{rank}(\hat{J}+P)$.
Next, writing $G=\Phi^\top\Phi$, then every member of $T_G\mathcal{M}_d^n$ can be expressed as $\Phi^\top E+E^\top\Phi$ for some $E\in\mathbb{R}^{d\times n}$ with $\operatorname{diag}(\Phi^\top E)=\mathbf{0}$.
As such, $Y\in N_G\mathcal{M}_d^n$ follows from
\[
\operatorname{tr}\big(Y(\Phi^\top E+E^\top\Phi)\big)
=2\operatorname{tr}(Y\Phi^\top E)
=2\left(\operatorname{tr}\big((\hat{J}+P)\Phi^\top E\big)-c\operatorname{tr}(\Phi^\top E)\right)
=0,
\]
where the last step applies the facts that $\hat{J}+P$ is the orthogonal projection onto the column space of $\Phi^\top$ and $\operatorname{diag}(\Phi^\top E)=\mathbf{0}$.
Next, a change of basis gives
\[
v(\alpha+\beta)(\hat{J}+P)
=(\beta v+\alpha-k)I+(v-k+\alpha)A-(k-\alpha)B.
\]
Since $v-k+\alpha$ and $k-\alpha$ are both strictly positive, we conclude that the off-diagonal entries of $Y=\hat{J}+P-cI$ are all nonzero and match the sign of the corresponding entries of $G$.
Overall, $Y$ satisfies the conditions in Lemma~\ref{lem.dual cert}(a)--(b), and furthermore, $S$ contains all pairs $(i,j)$ with $i\neq j$, and so $L$'s injectivity follows from the fact that every member of $T_G\mathcal{M}_d^n$ has zero diagonal.
\end{proof}

In the previous proof, our certificate $\hat{J}+P$ is the Gram matrix of a 2-distance tight frame, and such objects were recently studied in~\cite{BargGOY:15}.
While the previous result indicates that ETFs can be used to build smaller packings, the following result uses ETFs to build larger packings:

\begin{corollary}
\label{cor.lifted etfs}
Let $A$ and $B$ denote a $d\times 2d$ equiangular tight frame and its Naimark complement, and let $\theta$ denote any odd multiple of $\pi/8$.
Then the orthobiangular tight frame
\[
\Phi=\left[
\begin{array}{rr}
A\cos\theta & -A\sin\theta\\
B\sin\theta & B\cos\theta
\end{array}
\right]
\]
is a strongly locally optimal $4d$-packing in $\mathbb{R}\mathbf{P}^{2d-1}$.
\end{corollary}

\begin{proof}
For convenience, write $A=[a_1\cdots a_{2d}]$, $B=[b_1\cdots b_{2d}]$, and $\Phi=[u_1\cdots u_{2d}~v_1\cdots v_{2d}]$.
It is straightforward to verify that the rows of $\Phi$ are orthogonal with equal norm, and so $\Phi$ is tight.
Since $B$ is the Naimark complement of $A$, we have $\langle b_i,b_j\rangle=2\delta_{ij}-\langle a_i,a_j\rangle$.
With this identity, we derive the Gram matrix of $\Phi$:
$\langle u_i,u_i\rangle=\langle v_i,v_i\rangle=1$ and $\langle u_i,v_i\rangle=0$ for all $i\in[2d]$, and
\[
\langle u_i,u_j\rangle
=\langle a_i,a_j\rangle\cos2\theta,
\qquad
\langle v_i,v_j\rangle
=-\langle a_i,a_j\rangle\cos2\theta,
\qquad
\langle u_i,v_j\rangle
=-\langle a_i,a_j\rangle\sin2\theta
\]
for all $i,j\in[2d]$ with $i\neq j$.
Since $2\theta$ is an odd multiple of $\pi/4$, then $\left|\cos2\theta\right|=\left|\sin2\theta\right|=1/\sqrt{2}$, and so $\Phi$ is orthobiangular.

To demonstrate strong local optimality, Theorem~\ref{thm.strong local general} gives that it suffices to show $G=\Phi^\top\Phi$ is the unique global minimizer of \eqref{eq.primal}.
To this end, we will select $Y$ to be $\Phi^\top\Phi-I$ and apply Lemma~\ref{lem.dual cert}.
Indeed, $Y\neq0$ since $\Phi^\top\Phi$ has nontrivial off-diagonal.
Next, every member of $T_G\mathcal{M}_{2d}^{4d}$ can be expressed as $\Phi^\top E+E^\top\Phi$ for some $E\in\mathbb{R}^{2d\times 4d}$ with $\operatorname{diag}(\Phi^\top E)=\mathbf{0}$.
As such, $Y\in N_G\mathcal{M}_{2d}^{4d}$ follows from
\[
\operatorname{tr}\big(Y(\Phi^\top E+E^\top\Phi)\big)
=2\operatorname{tr}(Y\Phi^\top E)
=2\left(\operatorname{tr}(\Phi^\top\Phi\Phi^\top E)-\operatorname{tr}(\Phi^\top E)\right)
=0,
\]
where the last step applies the facts that $\Phi\Phi^\top=2I$ and $\operatorname{diag}(\Phi^\top E)=\mathbf{0}$.
Furthermore, $Y$ satisfies the conditions in Lemma~\ref{lem.dual cert}(a)--(b) since $Y_{ij}=G_{ij}-\delta_{ij}$ and $Y_{ij}=0$ precisely when $|Y_{ij}|<\|Y\|_\infty$.

It remains to verify that $L$ restricted to $T_G\mathcal{M}_{2d}^{4d}$ is injective.
To this end, taking $R=\{(i,j)\in[n]^2:Y_{ij}=0,~i<j\}$, we will show that $K\colon Z\mapsto\{Z_{ij}\}_{(i,j)\in R}$ restricted to $N_G\mathcal{M}_{2d}^{4d}$ is surjective.
To see why this suffices, pick any $X\in T_G\mathcal{M}_{2d}^{4d}$ such that $L(X)=0$.
Then decomposing over matrix entries gives
\begin{align*}
0
&=\langle X,Z\rangle\\
&=\langle \operatorname{diag} X,\operatorname{diag} Z\rangle+\langle K(X),K(Z)\rangle+\langle K(X^\top),K(Z^\top)\rangle+\langle L(X),L(Z)\rangle\\
&=2\langle K(X),K(Z)\rangle.
\end{align*}
for every $Z\in N_G\mathcal{M}_{2d}^{4d}$, and so $K(X)=0$, meaning every entry of $X$ is zero, i.e., $X=0$, as desired.
To prove surjectivity, it suffices to find matrices $Z_1,\ldots,Z_d\in N_G\mathcal{M}_{2d}^{4d}$ such that each $K(Z_k)$ is a signed version of the $k$th identity basis element.
In particular, we will find $z_k$ in the nullspace of $\Phi$ and put $Z_k=z_kz_k^\top$.
Indeed, write $\Phi=[U~V]$, let $\delta_k$ denote the $k$th identity basis element in $\mathbb{R}^{2d}$, and take $z_k=\delta_k\oplus -V^{-1}U\delta_k$. 
Then a short calculation gives that $K(Z_k)=-(V^{-1}U)_{kk}\delta_k$, and furthermore, the diagonal of $-V^{-1}U$ is constant $\cot2\theta\in\{\pm1\}$.
\end{proof}

\section{Packings from incidence structures}
\label{sec.incidence}

In this section, we use combinatorial designs to construct infinite families of near-optimal packings.
The low-dimensional instances of these infinite families appear in Sloane's database~\cite{Sloane:online} as putatively optimal, and we prove that all of these packings are optimal in a certain weak sense.

An \textbf{incidence structure} is a triple $C=(P,L,I)$, where $P$ is a set of points, $L$ is a set of lines, and $I$ is an incidence relation with the interpretation that $(p,l)\in I$ when the point $p\in P$ lies on the line $l\in L$.
The \textbf{dual structure} of $C=(P,L,I)$ is $C^*=(L,P,I^*)$, where $I^*=\{(l,p):(p,l)\in I\}$.
We say $C$ is $k$-\textbf{uniform} if $|l|=k$ for every $l\in L$.
We say $x$ is an \textbf{intersection number} of $C$ if there exist $l,l'\in L$ such that $|l\cap l'|=x$.
For each $l\in L$, there exists a \textbf{super embedding} $E_l\colon\mathbb{R}^{|l|}\rightarrow\mathbb{R}^l\subseteq\mathbb{R}^P$ such that $\|E_lz\|_q=\|z\|_q$ for every $z\in\mathbb{R}^{|l|}$ and $q\in[1,\infty]$; here, $\mathbb{R}^P$ denotes the vector space of real-valued functions over $P$.
For example, for any enumeration $l=\{p_1,\ldots,p_{|l|}\}$, one may take $E_l\delta_i=\pm\delta_{p_i}$ for each $i\in[|l|]$ and extend linearly; here, $\delta_k$ denotes the $k$th identity basis element in the appropriate vector space.
(In fact, one may apply the super embedding property for $q=2$ and $q=\infty$ to show that every super embedding has this form.)

\begin{lemma}
\label{lem.star product}
Take a $k$-uniform incidence structure $(P,L,I)$ with intersection numbers in $\{0,1\}$, along with vectors $\{v_j\}_{j\in J}$ in $\mathbb{R}^k$ such that $\|v_j\|_2^2=k$ and $\|v_j\|_\infty=1$ for all $j\in J$.
Then $\{E_lv_j\}_{l\in L,j\in J}$ in $\mathbb{R}^P$ satisfies
\[
|\langle E_lv_j,E_{l'}v_{j'}\rangle|
=\left\{
\begin{array}{cl}
|\langle v_j,v_{j'}\rangle|&\text{if }l=l'\\
|l\cap l'|&\text{if }l\neq l'
\end{array}
\right.
\]
\end{lemma}

\begin{proof}
Note that $\|Ex\|_2=\|x\|_2$ for all $x\in\mathbb{R}^k$ implies $E^*E=I_k$.
As such, when $l=l'$, we have 
\[
\langle E_lv_j,E_{l'}v_{j'}\rangle
=\langle E_lv_j,E_lv_{j'}\rangle
=\langle v_j,E_l^*E_lv_{j'}\rangle
=\langle v_j,v_{j'}\rangle.
\]
Otherwise, we may write out the inner product
\[
\langle E_lv_j,E_{l'}v_{j'}\rangle
=\sum_{p\in P}(E_lv_j)(p)(E_{l'}v_{j'})(p)
=\sum_{p\in l\cap l'}(E_lv_j)(p)(E_{l'}v_{j'})(p),
\]
which is $0$ whenever $l\cap l'$ is empty.
For the remaining case where $|l\cap l'|=1$, first note that our assumptions on each $v_j$ imply $\|E_lv_j\|_2^2=k$ and $\|E_lv_j\|_\infty=1$.
Since $E_lv_j\in\mathbb{R}^l$ and $|l|=k$, then $|(E_lv_j)(p)|=1$ for each $p\in l$.
Overall, denoting $p_0\in l\cap l'$, we have
\[
|\langle E_lv_j,E_{l'}v_{j'}\rangle|
=\bigg|\sum_{p\in l\cap l'}(E_lv_j)(p)(E_{l'}v_{j'})(p)\bigg|
=|(E_lv_j)(p_0)||(E_{l'}v_{j'})(p_0)|
=1.
\qedhere
\]
\end{proof}

As an application of Lemma~\ref{lem.star product}, suppose $C$ is the dual of a Steiner system with $t=2$.
Then $1$ is the only intersection number of $C$.
As such, if the $v_j$'s further satisfy $|\langle v_j,v_{j'}\rangle|=1$ for all $j\neq j'$, corresponding to the vertices of a regular simplex in $\mathbb{R}^k$, then $\{E_lv_j\}_{l\in L,j\in J}$ is equiangular.
Finally, the Welch bound gives that these equiangular vectors form an ETF when $k$ is large enough.
This is precisely the construction of Steiner ETFs~\cite{FickusMT:12}; see~\cite{Seidel:73} for Seidel's precursor construction in the context of 2-graphs.
A modification of this construction was recently used to produce so-called Tremain ETFs, which in turn led to new strongly regular graphs~\cite{FickusJMP:16}.
In this paper, we use Lemma~\ref{lem.star product} to construct new infinite families of putatively optimal line packings.

\begin{theorem}
\label{thm.opt obtfs}
Let $q$ be a prime power, and suppose there exists a Hadamard matrix $H$ of order $q+1$ with an all-ones row.
Let $H_-$ denote the $q\times(q+1)$ submatrix obtained by removing this all-ones row.
Then for each of the following pairs $(C,V)$ of incidence structures $C=(P,L,I)$ and vector ensembles $V=\{v_j\}_{j\in J}$ there exist super embeddings $\{E_l\}_{l\in L}$ such that $\{E_lv_j\}_{l\in L,j\in J}$ is an orthobiangular tight frame for its span:
\[
(\mathbf{P}_q,H),
\qquad
(\mathbf{P}_q,H_-^\top),
\qquad
(\mathbf{A}_q,H_-),
\qquad
(\mathbf{A}_q^*,H),
\]
where $\mathbf{P}_q$ and $\mathbf{A}_q$ denote the projective and affine planes of order $q$, respectively.
In each case, the coherence is $1+o(1)$ times Welch's lower bound, and there is no orthobiangular tight frame of the same size in the same space with smaller coherence.
\end{theorem}

We note that low-dimensional instances of these constructions are putatively optimal in Sloane's database; see Table~\ref{table}.
Our proof of the coherence-minimizing properties of these orthobiangular tight frames follows from certain integrality conditions, which we develop here.

\begin{lemma}
\label{lem.integrality}
Let $G$ be the Gram matrix of an orthobiangular tight frame of $n$ vectors in $\mathbb{R}^d$ with $n\neq2d$.
Then there exists an integer $z$ such that every column of $G$ contains exactly $z$ zeros and
\[
\sqrt{\frac{d(n-z-1)}{n-d}},
\qquad
\sqrt{\frac{(n-d)(n-z-1)}{d}}
\]
are both integers.
\end{lemma}

\begin{proof}
Without loss of generality, the diagonal of $G$ is all $1$s.
Write $G=\Phi^\top\Phi$ with $\Phi=[\varphi_1\cdots\varphi_n]$, let $\mu$ denote the coherence of $\Phi$, and let $z_i$ denote the number of zeros in the $i$th column of $G$.
Then the squared norm of the $i$th column of $G$ is
\[
1+(n-z_i-1)\mu^2
=\|\Phi^\top\varphi_i\|_2^2
=\frac{n}{d}\|\varphi_i\|_2^2
=\frac{n}{d},
\]
which is constant over $i\in[n]$.
Put $z=z_i$.
Then rearranging gives
\begin{equation}
\label{eq.mu from z}
\mu=\sqrt{\frac{n-d}{d(n-z-1)}}.
\end{equation}
For the integrality conditions, we apply Proposition~\ref{prop.int evals} to $S=(1/\mu)(G-I)$.
Indeed, the eigenvalue multiplicities match those of $G$, namely $d$ and $n-d$, which are distinct since $n\neq 2d$.
Also, the spectrum of $S$ is a shifted and scaled version of the spectrum $\{n/d,0\}$ of $G$, namely $(1/\mu)(n/d-1)$ and $-1/\mu$.
The result then follows by plugging in \eqref{eq.mu from z}.
\end{proof}

Judging by \eqref{eq.mu from z}, it is clear that the coherence $\mu$ of an orthobiangular tight frame is within $1+o(1)$ of the Welch bound precisely when $z=o(n)$.
In words, a vanishing fraction of the Gram matrix entries are $0$s, meaning the frame approaches equiangularity in some sense.

\begin{proof}[Proof of Theorem~\ref{thm.opt obtfs}]
In each case, $C$ has intersection numbers in $\{0,1\}$ and $V$ has inner products in $\{0,-1\}$, and so $\Phi=\{E_lv_j\}_{l\in L,j\in J}$ is orthobiangular by Lemma~\ref{lem.star product}.
We will show that $\Phi$ is tight.
Since each $\|E_lv_j\|_2^2=k$, it suffices to show $\|\Phi^\top\Phi\|_F^2=k^2n^2/d$.
In addition, to demonstrate that these are the orthobiangular tight frames of minimal coherence, \eqref{eq.mu from z} gives that it suffices to show that the integrality conditions in Lemma~\ref{lem.integrality} are violated whenever the Gram matrix has fewer zeros.
We proceed by considering each case individually.

\textbf{Case I:} $(\mathbf{P}_q,H)$.
The projective plane $\mathbf{P}_q$ has $q^2+q+1$ points and $q^2+q+1$ lines, each containing $k=q+1$ points.
As such, $\Phi$ amounts to $n=|L||J|=(q+1)(q^2+q+1)$ vectors in $d=q^2+q+1$ dimensions.
Since $1$ is the only intersection number, each column of $\Phi^\top\Phi$ contains a $k$ on the diagonal, $(|L|-1)|J|=q(q+1)^2$ different $\pm1$s on the off-diagonal, and the remaining $q$ entries are $0$s.
Overall, $\|\Phi^\top\Phi\|_F^2=n(k^2+q(q+1)^2)=k^2n^2/d$, and so $\Phi$ is tight, and furthermore $q=o(n)$ implies the coherence is $1+o(1)$ of the Welch bound.
Next, we claim the integrality conditions in Lemma~\ref{lem.integrality} are violated whenever $z<q$.
Indeed,
\begin{equation}
\label{eq.int1}
\frac{d(n-z-1)}{n-d}
=q^2+2q+2-\frac{z}{q},
\end{equation}
and so integrality requires $q$ to divide $z$.
Also $z\neq0$, since otherwise $\eqref{eq.int1}=(q+1)^2+1$ fails to be a perfect square.

\textbf{Case II:} $(\mathbf{P}_q,H_-^\top)$.
In this case, pick each $E_l$ so that $E_l^*\mathbf{1}=\mathbf{1}$.
Then since each $v_j$ is orthogonal to the all-ones vector, we have $\langle E_lv_j,\mathbf{1}\rangle=\langle v_j,E_l^*\mathbf{1}\rangle=\langle v_j,\mathbf{1}\rangle=0$.
As such, $\Phi$ amounts to $n=|L||J|=q(q^2+q+1)$ vectors in a hyperplane of dimension $d=q(q+1)$.
As before, $1$ is the only intersection number, and so each column of $\Phi^\top\Phi$ contains a $k$ on the diagonal, $(|L|-1)|J|=q^2(q+1)$ different $\pm1$s on the off-diagonal, and the remaining $q-1$ entries are $0$s.
Then $\|\Phi^\top\Phi\|_F^2=n(k^2+q^2(q+1))=k^2n^2/d$, and so $\Phi$ is tight, and furthermore $q-1=o(n)$ implies the coherence is $1+o(1)$ of the Welch bound.
Next, we show the integrality conditions are violated whenever $z<q-1$.
To this end,
\[
\frac{d(n-z-1)}{n-d}
=(q+1)^2+\frac{(q+1)(q-z-1)}{q^2},
\]
and so integrality requires $q^2$ to divide $q-z-1$.
This is not possible when $z<q-1$, since this implies $q^2>q-z-1>0$.

\textbf{Case III:} $(\mathbf{A}_q,H_-)$.
The affine plane $\mathbf{A}_q$ has $q^2$ points and $q(q+1)$ lines, each containing $k=q$ points.
As such, $\Phi$ amounts to $n=|L||J|=q(q+1)^2$ vectors in $d=q^2$ dimensions.
For this incidence structure, lines intersect unless they are parallel, and each line is parallel to $q-1$ other lines.
Each column of $\Phi^\top\Phi$ therefore contains a $k$ on the diagonal, $q^2-1$ different $0$s on the off-diagonal (each coming from a parallel line), and the remaining $q(q^2+q+1)$ entries are $\pm1$s.
Overall, $\|\Phi^\top\Phi\|_F^2=n(k^2+q(q^2+q+1))=k^2n^2/d$, and so $\Phi$ is tight, and furthermore $q^2-1=o(n)$ implies the coherence is $1+o(1)$ of the Welch bound.
Next, we show the integrality conditions are violated whenever $z<q^2-1$.
To this end,
\[
\frac{d(n-z-1)}{n-d}
=q(q+1)-\frac{q(z+q+2)}{q^2+q+1},
\]
and so integrality requires $q^2+q+1$ to divide $z+q+2$.
This is not possible when $z<q^2-1$, since this implies $q^2+q+1>z+q+2>0$.

\textbf{Case IV:} $(\mathbf{A}_q^*,H)$.
The dual $\mathbf{A}_q^*$ of the affine plane has $q(q+1)$ points and $q^2$ lines, each containing $k=q+1$ points.
As such, $\Phi$ amounts to $n=|L||J|=q^2(q+1)$ vectors in $d=q(q+1)$ dimensions.
Since two points in $\mathbf{A}_q$ determine a line, $1$ is the only intersection number of $\mathbf{A}_q^*$, and so each column of $\Phi^\top\Phi$ contains a $k$ on the diagonal, $(|L|-1)|J|=(q-1)(q+1)^2$ different $\pm1$s on the off-diagonal, and the remaining $q$ entries are $0$s.
Overall, $\|\Phi^\top\Phi\|_F^2=n(k^2+(q-1)(q+1)^2)=k^2n^2/d$, and so $\Phi$ is tight, and furthermore $q=o(n)$ implies the coherence is $1+o(1)$ of the Welch bound.
Next, we show the integrality conditions are violated whenever $z<q$.
\begin{equation}
\label{eq.int2}
\frac{d(n-z-1)}{n-d}
=q^2+2q+2-\frac{z-1}{q-1},
\end{equation}
and so integrality requires $q-1$ to divide $z-1$.
Also, $z\neq 1$, since otherwise $\eqref{eq.int2}=(q+1)^2+1$ fails to be a perfect square.
\end{proof}

\section{Sporadic packings}
\label{sec.sporadic}

Every known infinite family of optimal packings in $\mathbb{R}\mathbf{P}^{d-1}$ with $d>2$ achieves equality in either the Welch bound or the orthoplex bound~\cite{FickusM:15,BoykinSTW:05,BodmannH:16}.
Since these packings are all tight frames with small angle sets, we were encouraged to investigate the packings in Sloane's database~\cite{Sloane:online} that share these features (to within numerical precision).
In this section, we describe perfected versions the packings that are not yet known to belong to an infinite family.
Our hope is that these descriptions might replicate how Tremain's notes~\cite{Tremain:09} inspired the infinite families in~\cite{FickusMT:12,FickusJMP:16}.
We index each description by the corresponding packing parameters $(d,n)$.

\subsection{Classical packings}

Each of the $n$-packings in this subsection arise from an antipodal spherical code of $2n$ points by collecting antipodal representatives, much like how the optimal 6-packing in $\mathbb{R}\mathbf{P}^2$ is obtained from the 12 vertices of the icosahedron.

\textbf{(3,12)}
The $24$ vertices of the rhombicuboctahedron may be obtained by taking all even permutations of $(\pm1,\pm1,\pm(1+\sqrt{2}))$.
The corresponding $12$-packing is putatively optimal.

\textbf{(4,60)}
The 600-cell has $16$ vertices of the form $(\pm1,\pm1,\pm1,\pm1)$, 8 vertices obtained from all permutations of $(\pm2,0,0,0)$, and 96 vertices obtained from all even permutations of $(\pm\phi,\pm1,\pm1/\phi,0)$, where $\phi=(1+\sqrt{5})/2$ is the golden ratio.
The corresponding 60-packing can be proved optimal using Delsarte's linear program~\cite{AndreevNN:99,HaasHM:17}.

\textbf{(5,20)}
Let $C=(P,L,I)$ denote the incidence structure corresponding to the complete graph on $5$ vertices, and let $H=[h_1~h_2]$ denote a Hadamard matrix of order $2$.
Then for any choice of super embeddings $\{E_l\}_{l\in L}$, the ensemble $\{E_lh_j\}_{l\in L,j\in[2]}$ forms a putatively optimal packing.
These are the shortest nonzero vectors in the $D_5$ root lattice (Henry Cohn graciously pointed this out to the authors).

\textbf{(6,36)/(7,63)/(8,120)}
The shortest nonzero vectors in the lattice $E_8\subseteq\mathbb{R}^8$ have norm $\sqrt{2}$.
There are 240 shortest nonzero vectors: 112 have the form $\pm\delta_i\pm\delta_j$ for $i,j\in[8]$ with $i\neq j$, and 128 have the from $(\pm1/2,\ldots,\pm1/2)$ with an even number of minus signs.
Intersecting $E_8$ with the orthogonal complement of $\mathbf{1}$ produces the 7-dimensional lattice $E_7$.
There are 56 shortest nonzero vectors of the form $\pm\delta_i\mp\delta_j$ and 70 of the form $(\pm1/2,\ldots,\pm1/2)$ with exactly 4 minus signs, totaling 126 vectors.
Intersecting $E_7$ with the orthogonal complement of $\delta_1+\delta_2$ produces the 6-dimensional lattice $E_6$, which has 72 shortest nonzero vectors: 32 of the from $\pm\delta_i\mp\delta_j$ (where either $i,j\leq 2$ or $i,j>2$) and 40 of the form $(\pm1/2,\ldots,\pm1/2)$ with exactly 4 minus signs, exactly one of which is in the first two coordinates.
In all three cases, the corresponding packing achieves equality in Levenshtein's bound~\cite{Levenshtein:92,HaasHM:17}.

\subsection{Marriage packings}

In certain special settings, one may combine optimal packings to produce larger optimal packings.
Known examples include mutually unbiased bases~\cite{BoykinSTW:05} and the packings in~\cite{BodmannH:16}.
Such ``marriage'' packings are delicate because the sub-packings must interact well for the construction to work.

\textbf{(3,7)}
Take any $3\times 4$ submatrix of a Hadamard matrix of order 4.
Appending the permutations of $(\sqrt{3},0,0)$ produces a 7-packing that achieves equality in the orthoplex bound~\cite{BodmannH:16}.
Modulo rotation, this is the unique optimal 7-packing in $\mathbb{R}\mathbf{P}^2$~\cite{CohnW:12}.

\textbf{(5,16)}
Embed a $5\times 10$ ETF in $\mathbb{R}^6$ by taking all permutations of $\sqrt{5}(1,1,1,-1,-1,-1)$ and selecting antipodal representatives.
Combining with a lifted simplex, specifically, the permutations of $(5,-1,-1,-1,-1,-1,-1)$, produces a 16-packing that achieves equality in the orthoplex bound.

\textbf{(6,22)}
Take the $6\times 16$ ETF that arises from selecting rows the Hadamard transform over $(\mathbb{Z}/2\mathbb{Z})^4$ according to the McFarland difference set~\cite{FickusM:15}.
Each vector in this ETF has all $\pm1$ entries.
Combining these with the permutations of $(\sqrt{6},0,0,0,0,0)$ produces a 22-packing that achieves equality in the orthoplex bound~\cite{BodmannH:16}.

\textbf{(6,63)/(7,91)}
The $7\times 28$ ETF enjoys a natural embedding in $\mathbb{R}^8$, namely, taking all permutations of $x=(3,3,-1,-1,-1,-1,-1,-1)$.
Scaling these vectors by $1/\sqrt{3}$ and combining with the packing associated with $E_7$ produces a putatively optimal packing.
Next, remove $x$ and project the other ETF vectors onto the orthogonal complement of $x$.
Then after scaling these 27 vectors appropriately, they can be combined with the packing associated with $E_6$ to produce a putatively optimal packing.

\subsection{Misfit packings}

The following packings are too peculiar to be associated with the previous constructions, and so we quarantined them into this final subsection.

\textbf{(6,24)}
Put $a=(2,2,2,2)$, $b=(2,-2,-1,1)$ and $c=(1,-1,2,-2)$.
Then the $24$ columns of
\[
\left[
\begin{array}{rrrrrr}
a&0&0&0&b&b\\
0&a&0&0&\phantom{-}c&-c\\
b&b&a&0&0&0\\
c&-c&0&a&0&0\\
0&0&b&b&a&0\\
0&0&\phantom{-}c&-c&0&a
\end{array}
\right]
\]
form a putatively optimal packing.

\textbf{(7,36)}
In an earlier version of this manuscript, we wrote that this packing eludes us.
In response, Henry Cohn supplied us with a construction:
The group $SL(2,8)$ has four $7$-dimensional irreducible complex representations, exactly one of each is defined over the reals.
There are $36$ different $7$-Sylow subgroups in $SL(2,8)$, each fixing a unique line in $\mathbb{R}^7$ through this representation.
This packing of 36 lines is putatively optimal.
See~\cite{IversonJM:17} for an alternate description of this construction.

Let $\{\varphi_i\}_{i\in[36]}$ denote the unit vectors in the packing.
Then $\langle\varphi_i,\varphi_j\rangle\in\{1,\pm1/7,\pm3/7\}$ for all $i,j\in[36]$, i.e., the packing is biangular.
Strangely, if we put
\[
\Phi_i=\varphi_i\varphi_i^\top-\frac{3+\sqrt{2}}{21}I,
\]
then $\{\Phi_i\}_{i\in[36]}$ forms a tight frame for its span, the Naimark complement of which is the $(8,36)$ construction below.

\textbf{(8,32)}
Let $C=(P,L,I)$ denote the unique balanced incomplete block design with parameters $(9,3,1)$; see Example~1.22 in~\cite{MathonR:07}.
Fix $p\in P$ and define $C'=(P',L',I')$ by $P'=P\setminus\{p\}$, $L'=\{l\in L:p\not\in l\}$, $I'=\{(p,l)\in I:p\in P',l\in L'\}$.
Let $H_-=[v_1~v_2~v_3~v_4]$ denote any $3\times4$ submatrix of a Hadamard matrix of order $4$.
Then for any choice of super embeddings $\{E_l\}_{l\in L'}$, the ensemble $\{E_lv_j\}_{l\in L',j\in[4]}$ forms a putatively optimal packing.

\textbf{(8,36)}
The adjacency matrix of any $(36,14,7,4)$-strongly regular graph has an eigenspace of dimension $8$.
The orthogonal projection onto this eigenspace is the Gram matrix of a $2$-distance tight frame \`{a} la~\cite{BargGOY:15}, and is putatively optimal.

\textbf{(10,40)}
Let $C=(P,L,I)$ be the incidence structure whose points are the vertices of the Peterson graph $G$, and whose lines, also indexed by vertices $v\in V(G)$, are comprised of $v$ and its neighborhood in $G$.
Letting $A$ denote the adjacency matrix of $G$, put $2A+I=[m_1\cdots m_{10}]$, and let $H=[h_1\cdots h_4]$ be a Hadamard matrix of order 4.
Then for any choice of super embeddings $\{E_{l_v}\}_{v\in V(G)}$, the ensemble $\{\operatorname{diag}(m_v)E_{l_v}h_j\}_{v\in V(G),j\in[4]}$ forms a putatively optimal packing.
One may interpret this as a generalization of Theorem~\ref{thm.opt obtfs} in which a balanced ternary design plays the role of an incidence structure, specifically, the design given in Example 2.5 of~\cite{HurdS:07}.

\section*{Acknowledgments}

The authors thank Prof.\ Henry Cohn and the anonymous referees for extensive feedback that greatly improved the presentation of both our results and the relevant literature.
This work was partially supported by NSF DMS 1321779, AFOSR F4FGA06060J007, AFOSR Young Investigator Research Program award F4FGA06088J001, and an AFRL Summer Faculty Fellowship Program award.
The views expressed in this article are those of the authors and do not reflect the official policy or position of the United States Air Force, Department of Defense, or the U.S.\ Government.

\appendix

\section{Summary of low-dimensional results}

Table~\ref{table} gives a summary of the low-dimensional instances of our results (along with ETFs from~\cite{FickusM:15} and mutually unbiased bases from~\cite{BoykinSTW:05} for reference).
In each case, the coherence under ``$\mu$'' is rounded up to the next multiple of $10^{-4}$, and for precision, we also report the coherence's minimal polynomial over $\mathbb{Q}$ (we cleared the denominators in favor of integer coefficients).
Next, ``opt'' indicates optimality: C denotes computer-assisted proof, W denotes the Welch bound, O denotes the orthoplex bound, D denotes Delsarte's linear programming bound, and L denotes local optimality.
Starred rows provide substantial improvements over the corresponding packings in Sloane's database~\cite{Sloane:online}.
Finally, we list the number of angles in the packing, whether the packing is a tight frame, as well as some brief notes (such as ``ETF'' if the packing is an equiangular tight frame, or ``equiangular'' if all of the off-diagonal entries in the Gram matrix have the same absolute value).
These notes are not intended to be complete descriptions of the packings; see the referenced ``location'' for more information.

\begin{table}
\begin{center}
\footnotesize{
\begin{tabular}{rrcrcccll}
$d$ & $n$ & $\mu$ & min polynomial & opt & angles & tight & notes & location \\  \hline
3 & 5 & 0.4473 & $5x^2-1$ & C & 1 & - & equiangular & Sec.~\ref{sec.cad} \\
3 & 6 & 0.4473 &  $5x^2-1$& W & 1 & + & ETF & Ref.~\cite{FickusM:15} \\
3 & 7 & 0.5774 & $3x^2-1$ & O & 3 & + & marriage & Sec.~\ref{sec.sporadic}\\
3 & 12 & 0.7446 & $17x^2-14x+1$ &  & 3 & + & rhombicuboctahedron & Sec.~\ref{sec.sporadic} \\
4 & 6   & 0.3334 & $3x-1$  & C & 1 & - & equiangular &  Thm.~\ref{thm.6in4} \\
4 & 12 & 0.5000 & $2x-1$ & O & 2 & + & mutually unbiased bases & Ref.~\cite{BoykinSTW:05} \\
4 & 60 & 0.8091 & $4x^2-2x-1$ & D & 4 & + & 600-cell & Sec.~\ref{sec.sporadic} \\
5 & 7   & 0.2863 & $x^3-9x^2-x+1$ &     & 1 & - & provably optimal? & Eq.~\eqref{eq.conjectured grams} \\
5 & 10 & 0.3334 & $3x-1$ & W & 1 & + & ETF & Ref.~\cite{FickusM:15} \\
5 & 16 & 0.4473 & $5x^2-1$ & O & 3 & + & marriage & Sec.~\ref{sec.sporadic} \\
5 & 20 & 0.5000 & $2x-1$ &  & 2 & + & $D_5$ & Sec.~\ref{sec.sporadic} \\
6 & 8   & 0.2410 & Eq.~\eqref{eq.6by8coherence} & & 2 & - & provably optimal? & Eq.~\eqref{eq.conjectured grams}  \\
6 & 12 & 0.3163 & $10x^2-1$ & L & 2 & + & lifted ETF & Cor.~\ref{cor.lifted etfs}\\
6 & 15 & 0.3334 & $3x-1$ & L & 1 & - & srg(15,8,4,4) & Cor.~\ref{cor.srg} \\
6 & 16 & 0.3334 & $3x-1$ & W & 1 & + & ETF & Ref.~\cite{FickusM:15} \\
6 & 22 & 0.4083 & $6x^2-1$ & O & 3 & + & marriage & Sec.~\ref{sec.sporadic} \\
6 & 24 & 0.4445 & $9x-4$ & & 4 & + & misfit & Sec.~\ref{sec.sporadic} \\
6 & 36 & 0.5000 & $2x-1$ & D & 2 & + & $E_6$ & Sec.~\ref{sec.sporadic} \\
6 & 63 & 0.6124 & $8x^2-3$ & & 4 & + & marriage & Sec.~\ref{sec.sporadic} \\
7 & 9   & 0.2000 & $5x-1$ &  & 1 & - & equiangular & Sec.~\ref{sec.cad}  \\
7 & 10 & 0.2361 & $x^2+4x-1$  & & 1 & - & equiangular & Sec.~\ref{sec.cad} \\
7 & 14 & 0.2774 & $13x^2-1$ & & 1 & + & ETF & Ref.~\cite{FickusM:15} \\
7 & 27 & 0.3334 & $3x-1$ & L & 1 & - & srg(27,16,10,8) & Cor.~\ref{cor.srg} \\
7 & 28 & 0.3334 & $3x-1$ & W & 1 & + & ETF & Ref.~\cite{FickusM:15} \\
7 & 36 & 0.4286 & $7x-3$ &  & 2 & + & misfit & Sec.~\ref{sec.sporadic} \\
7 & 63 & 0.5000 & $2x-1$ & D & 2 & + & $E_7$ & Sec.~\ref{sec.sporadic} \\
7 & 91 & 0.5774 & $3x^2-1$ & & 4 & + & marriage & Sec.~\ref{sec.sporadic} \\
8 & 10 & 0.1828 & $19x^2+2x-1$ & & 1 & - & equiangular & Sec.~\ref{sec.cad} \\
8 & 32 & 0.3334 & $3x-1$ & & 2 & + & misfit & Sec.~\ref{sec.sporadic} \\
8 & 36 & 0.3572 & $14x-5$ & & 2 & + & misfit & Sec.~\ref{sec.sporadic} \\
8 & 120&0.5000 & $2x-1$ & D & 2 & + & $E_8$  & Sec.~\ref{sec.sporadic} \\
9 & 12 & 0.1828 & $19x^2+2x-1$ & & 1 & - & equiangular & Sec.~\ref{sec.cad} \\
9 & 18 & 0.2426 & $17x^2-1$ & W & 1 & + & ETF & Ref.~\cite{FickusM:15} \\
9 & 48 & 0.3334 & $3x-1$ & O & 2 & + & $(\mathbf{A}_2,H_-)$ & Thm.~\ref{thm.opt obtfs} \\
10& 12& 0.1429 & $7x-1$ &  & 1 & - & equiangular & Sec.~\ref{sec.cad} \\
10& 16& 0.2000 & $5x-1$ & W & 1 & + & ETF & Ref.~\cite{FickusM:15} \\
10& 20& 0.2358 & $18x^2-1$ & L & 2 & + & lifted ETF & Cor.~\ref{cor.lifted etfs} \\
10& 40& 0.3077 & $13x-4$ & & 3 & + & misfit & Sec.~\ref{sec.sporadic} \\
11& 14& 0.1578 & $x^3+21x^2+3x-1$ & & 1 & - & equiangular & Sec.~\ref{sec.cad} \\
11& 16& 0.1784 & $9x^2+4x-1$ & & 1 & - & equiangular & Sec.~\ref{sec.cad} \\
11& 18& 0.2000 & $5x-1$ & & 1 & - & equiangular & Sec.~\ref{sec.cad} \\
12& 36& 0.2500 & $4x-1$ &  & 2 & + & $(\mathbf{A}_3^*,H)$ & Thm.~\ref{thm.opt obtfs} \\
12& 39& 0.2500 & $4x-1$ &  & 2 & + & $(\mathbf{P}_3,H_-^\top)$ & Thm.~\ref{thm.opt obtfs} \\
13& 15& 0.1112 & $9x-1$ & & 1 & - & equiangular & Sec.~\ref{sec.cad} \\
13& 18& 0.1590 & $27x^2+2x-1$ & & 1 & - & equiangular & Sec.~\ref{sec.cad} \\
13& 19& 0.1663 & $31x^3+25x^2+x-1$ & & 1 & - & equiangular & Sec.~\ref{sec.cad} \\
13& 26& 0.2000 & $5x-1$ & W & 1 & + & ETF & Ref.~\cite{FickusM:15} \\
13& 52& 0.2500 & $4x-1$ &  & 2 & + & $(\mathbf{P}_3,H)$ & Thm.~\ref{thm.opt obtfs} \\
15& 18& 0.1149 & $41x^2+4x-1$ & & 1 & - & equiangular & Sec.~\ref{sec.cad} \\
15& 21& 0.1429 & $7x-1$ & & 1 & - & equiangular & Sec.~\ref{sec.cad} \\
*15& 30& 0.1857 & $29x^2-1$ & W & 1 & + & ETF & Ref.~\cite{FickusM:15} \\
*15& 35& 0.2000 & $5x-1$ & L & 1 & - & srg(35,18,9,9) & Cor.~\ref{cor.srg} \\
*15& 36& 0.2000 & $5x-1$ & W & 1 & + & ETF & Ref.~\cite{FickusM:15} \\
16& 18& 0.0910 & $11x-1$ & & 1 & - & equiangular & Sec.~\ref{sec.cad} \\
16& 23& 0.1429 & $7x-1$ & & 1 & - & equiangular & Sec.~\ref{sec.cad} \\
*16& 40& 0.2000 & $5x-1$ & L & 1 & - & srg(40,27,18,18) & Cor.~\ref{cor.srg} 
\end{tabular}
}
\end{center}
\caption{\label{table}
{\small
Summary of low-dimensional results.
}
}
\end{table}

\section{Derivation of the dual program \eqref{eq.dual}}

For convenience, we express members of the $\binom{n+1}{2}$-dimensional vector space of real symmetric $n\times n$ matrices in terms of the orthonormal basis $\{\delta_{ij}\}_{i,j\in[n],i\leq j}$ where $\delta_{ij}$ has a $1$ at entry $(i,j)$ and $0$'s elsewhere.
Let $L$ denote the linear operator that maps a symmetric matrix to the column vector of coordinates in this basis.
Let $A$ denote the matrix whose rows provide the coordinates for an orthonormal basis of $N_G\mathcal{M}_d^n$.
Then \eqref{eq.primal} can be re-expressed as
\[
\text{minimize}\quad\|x\|_\infty\quad\text{subject to}\quad Ax=b,
\]
where $x=L(X-I)$ and $b=AL(G-I)$.
Following Section~5 in~\cite{BoydV:04}, the Lagrangian is given by
\[
L(x,\nu)
:=\|x\|_\infty+\nu^\top(Ax-b),
\]
and so the dual program is
\[
g(\nu)
:=\inf_xL(x,\nu)
=\inf_x\Big[\|x\|_\infty+\nu^\top(Ax-b)\Big]
=\inf_x\Big[\|x\|_\infty+ \nu^\top Ax\Big]-\nu^\top b.
\]
At this point, H\"{o}lder's inequality gives $\nu^\top Ax\geq -\|A^\top\nu\|_1\|x\|_\infty$, with equality precisely when the entries of $x$ all satisfy
\[
|x_i|=\|x\|_\infty
\qquad
\text{and}
\qquad
(A^\top\nu)_ix_i\leq0.
\]
In particular, for every $x$, there exists $y$ such that $\|y\|_\infty=\|x\|_\infty$ and $\nu^\top Ay=-\|A^\top\nu\|_1\|x\|_\infty$, leading to the following simplification:
\[
g(\nu)
=\inf_x\bigg[\Big(1-\|A^\top\nu\|_1\Big)\|x\|_\infty\bigg]-\nu^\top b
=\left\{\begin{array}{ll}-b^\top\nu&\text{if }\|A^\top\nu\|_1\leq 1\\-\infty&\text{else.}\end{array}\right.
\]
As such, recalling the definition of $b$, the dual program $\max_\nu g(\nu)$ is equivalent to
\[
\text{minimize}\quad \big(L(G-I)\big)^\top\big(A^\top\nu\big)\quad\text{subject to}\quad\|A^\top\nu\|_1\leq1,
\]
which can be re-expressed as \eqref{eq.dual}.


\begin{thebibliography}{00}

\bibitem{AbsilMS:08}
P.-A.\ Absil, R.\ Mahony, R.\ Sepulchre,
Optimization Algorithms on Matrix Manifolds,
Princeton University Press, Princeton, NJ, 2008.

\bibitem{AndreevNN:99}
N.\ N.\ Andreev,
A spherical code,
Russ.\ Math.\ Surv.\ 54 (1999) 251--253.

\bibitem{ApplebyFMY:17}
M.\ Appleby, S.\ Flammia, G.\ McConnell, J.\ Yard,
SICs and algebraic number theory,
Found.\ Phys.\ 47 (2017) 1042--1059.

\bibitem{AzarijaM:15}
J.\ Azarija, T.\ Marc,
There is no $(75, 32, 10, 16)$ strongly regular graph,
Available online: arXiv:1509.05933

\bibitem{AzarijaM:16}
J.\ Azarija, T.\ Marc,
There is no $(95, 40, 12, 20)$ strongly regular graph,
Available online: arXiv:1603.02032

\bibitem{BallingerBCGKS:09}
B.\ Ballinger, G.\ Blekherman, H.\ Cohn, N.\ Giansiracusa, E.\ Kelly, A.\ Sch\"{u}rmann,
Experimental study of energy-minimizing point configurations on spheres,
Experiment.\ Math.\ 18 (2009) 257--283.

\bibitem{BandeiraFMW:13}
A.\ S.\ Bandeira, M.\ Fickus, D.\ G.\ Mixon, P.\ Wong,
The road to deterministic matrices with the restricted isometry property,
J.\ Fourier Anal.\ Appl.\ 19 (2013) 1123--1149.

\bibitem{BandeiraMM:17}
A.\ S.\ Bandeira, D.\ G.\ Mixon, J.\ Moreira,
A conditional construction of restricted isometries,
Int.\ Math.\ Res.\ Notices 2017 (2017) 372--381.

\bibitem{BaraniukDDW:08}
R.\ Baraniuk, M.\ Davenport, R.\ DeVore, M.\ Wakin,
A simple proof of the restricted isometry property for random matrices,
Constr.\ Approx.\ 28 (2008) 253--263.

\bibitem{BargGOY:15}
A.\ Barg, A.\ Glazyrin, K.\ Okoudjou, W-H.\ Yu,
Finite two-distance tight frames,
Linear Algebra Appl.\ 475 (2015) 163--175.

\bibitem{BenedettoK:06}
J.\ J.\ Benedetto, J.\ D.\ Kolesar,
Geometric properties of Grassmannian frames for $\mathbb{R}^2$ and $\mathbb{R}^3$,
EURASIP J.\ Appl.\ Signal Process.\ 2006 (2006) 1--17.

\bibitem{BochnakCR:98}
J.\ Bochnak, M.\ Coste, M.-F.\ Roy,
Real Algebraic Geometry,
Springer, 1998.

\bibitem{BodmannH:16}
B.\ G.\ Bodmann, J.\ Haas,
Achieving the orthoplex bound and constructing weighted complex projective 2-designs with Singer sets,
Linear Algebra Appl.\ 511 (2016) 54--71.

\bibitem{BourgainDFKK:11}
J.\ Bourgain, S.\ Dilworth, K.\ Ford, S.\ Konyagin, D.\ Kutzarova,
Explicit constructions of RIP matrices and related problems,
Duke Math.\ J.\ 159 (2011) 145--185.

\bibitem{BoydV:04}
S.\ Boyd, L.\ Vandenberghe,
Convex optimization,
Cambridge U.\ Press, 2004.

\bibitem{BoykinSTW:05}
P.\ O.\ Boykin, M.\ Sitharam, M.\ Tarifi, P.\ Wocjan,
Real Mutually Unbiased Bases,
Available online: arXiv:quant-ph/0502024

\bibitem{Brouwer:online}
A.\ E.\ Brouwer,
Parameters of Strongly Regular Graphs,
\url{https://www.win.tue.nl/~aeb/graphs/srg/srgtab.html}

\bibitem{BussemakerMS:81}
F.\ C.\ Bussemaker, R.\ A.\ Mathon, J.\ J.\ Seidel, Tables of two-graphs,
in:\ Combinatorics and Graph Theory, S.\ B.\ Rao, ed., Springer, (1981) 70--112.

\bibitem{CandesRT:06}
E.\ J.\ Cand\`{e}s, J.\ Romberg, T.\ Tao,
Robust uncertainty principles:\ Exact signal reconstruction from highly incomplete frequency information,
IEEE Trans.\ Inform.\ Theory 52 (2006) 489--509.

\bibitem{CasazzaF:06}
P.\ G.\ Casazza, M.\ Fickus,
Fourier transforms of finite chirps,
EURASIP J.\ Appl.\ Signal Process. 2006 (2006), 70204.

\bibitem{CasazzaK:03}
P.\ G.\ Casazza, J.\ Kova\v{c}evi\'{c},
Equal-norm tight frames with erasures,
Adv.\ Comput.\ Math.\ 18 (2003) 387--430.

\bibitem{Casselman:04}
B.\ Casselman,
The difficulties of kissing in three dimensions,
Notices Amer.\ Math.\ Soc.\ 51 (2004) 884--885.

\bibitem{ClarkC:81}
G.\ C.\ Clark, J.\ B.\ Cain, 
Error-Correction Coding for Digital Communications,
New York, Plenum Press, 1981.

\bibitem{CohnHM:16}
H.\ Cohn, A.\ Kumar, G.\ Minton,
Optimal simplices and codes in projective spaces,
Geom.\ Topol.\ 20 (2016) 1289--1357.

\bibitem{CohnW:12}
H.\ Cohn, J.\ Woo,
Three-point bounds for energy minimization,
J.\ Am.\ Math.\ Soc.\ 25 (2012) 929--958.

\bibitem{Collins:75}
G.\ E.\ Collins, 
Quantifier elimination for real closed fields by cylindrical algebraic decomposition,
in:\ Proc.\ 2nd GI Conference on Automata Theory and Formal Languages, Springer (1975) 134--183.

\bibitem{ConwayHS:96}
J.\ H.\ Conway, R.\ H.\ Hardin, N.\ J.\ A.\ Sloane,
Packing lines, planes, etc.: packings in Grassmannian spaces,
Experiment.\ Math.\ 5 (1996) 139--159.

\bibitem{CoutinhoGSZ:16}
G.\ Coutinho, C.\ Godsil, H.\ Shirazi, H.\ Zhan,
Equiangular lines and covers of the complete graph,
Linear Algebra Appl.\ 488 (2016) 264--283.

\bibitem{DhillonHST:08}
I.\ S.\ Dhillon, R.\ W.\ Heath, T.\ Strohmer, J.\ A.\ Tropp,
Constructing packings in Grassmannian manifolds via alternating projection,
Experiment.\ Math.\ 17 (2008): 9--35.

\bibitem{DingF:07}
C.\ Ding, T.\ Feng,
A generic construction of complex codebooks meeting the Welch bound,
IEEE Trans.\ Inform.\ Theory 53 (2007) 4245--4250.

\bibitem{Donoho:06}
D.\ L.\ Donoho,
Compressed sensing,
IEEE Trans.\ Inform.\ Theory 52 (2006) 1289--1306.

\bibitem{DonohoE:03}
D.\ L.\ Donoho, M.\ Elad,
Optimally sparse representation in general (nonorthogonal) dictionaries via $\ell^1$ minimization,
Proc.\ Natl.\ Acad.\ Sci.\ 100 (2003) 2197--2202.

\bibitem{DurtEBZ:10}
T.\ Durt, B.-G.\ Englert, I.\ Bengtsson, K.\ \.{Z}yczkowski,
On mutually unbiased bases,
Int.\ J.\ Quantum Inf.\	 8 (2010) 535--640.

\bibitem{FanLL:16}
J.\ Fan, Y.\ Liao, H.\ Liu,
An overview of the estimation of large covariance and precision matrices,
Econom.\ J.\ 19 (2016) C1--C32

\bibitem{FickusJ:18}
M.\ Fickus, J.\ Jasper,
Equiangular tight frames from group divisible designs,
in progress.

\bibitem{FickusJM:17a}
M.\ Fickus, J.\ Jasper, D.\ G.\ Mixon,
Mathematica-assisted proof of optimal $6$-packing in $\mathbb{R}\mathbf{P}^3$,
\url{https://www.dropbox.com/s/ahu52ns1iefo0vg/proof6in4.txt}

\bibitem{FickusJM:17b}
M.\ Fickus, J.\ Jasper, D.\ G.\ Mixon,
Mathematica-assisted proof of optimal $5$-packing in $\mathbb{R}\mathbf{P}^2$,
\url{https://www.dropbox.com/s/c01fnznc5brq4ra/proof5in3.txt}

\bibitem{FickusJMP:16}
M.\ Fickus, J.\ Jasper, D.\ G.\ Mixon, J.\ Peterson,
Tremain equiangular tight frames,
Available online: arXiv:1602.03490

\bibitem{FickusJMPW:16}
M.\ Fickus, J.\ Jasper, D.\ G.\ Mixon, J.\ Peterson, C.\ E.\ Watson,
Polyphase equiangular tight frames and abelian generalized quadrangles,
Available online: arXiv:1604.07488

\bibitem{FickusM:15}
M.\ Fickus, D.\ G.\ Mixon,
Tables of the existence of equiangular tight frames,
Available online: arXiv:1504.00253

\bibitem{FickusMJ:16}
M.\ Fickus, D.\ G.\ Mixon, J.\ Jasper,
Equiangular tight frames form hyperovals,
IEEE Trans.\ Inform.\ Theory 62 (2016) 5225--5236.

\bibitem{FickusMT:12}
M.\ Fickus, D.\ G.\ Mixon, J.\ C.\ Tremain,
Steiner equiangular tight frames,
Linear Algebra Appl.\ 436 (2012) 1014--1027.

\bibitem{Flammia:online}
S.\ Flammia,
Exact SIC fiducial vectors,
\url{http://www.physics.usyd.edu.au/~sflammia/SIC/}

\bibitem{FuchsHS:17}
C.\ A.\ Fuchs, M.\ C.\ Hoang, B.\ C.\ Stacey,
The SIC question:\ History and state of play,
Available online: arXiv:1703.07901

\bibitem{FuchsS:11}
C.\ A.\ Fuchs, R.\ Schack,
A quantum-Bayesian route to quantum-state space,
Found.\ Phys.\ 41 (2011) 345--356.

\bibitem{GodinhoN:14}
L.\ Godinho, J.\ Nat\'{a}rio,
An Introduction to Riemannian Geometry:\ With Applications to Mechanics and Relativity,
Springer, 2014.

\bibitem{GoyalKK:01}
V.\ K.\ Goyal, J.\ Kova\v{c}evi\'{c}, J.\ A.\ Kelner,
Quantized frame expansions with erasures,
Appl.\ Comput.\ Harmon.\ Anal.\ 10 (2001) 203--233.

\bibitem{Grey:62}
L.\ D.\ Grey,
Some bounds for error-correcting codes,
IRE Trans.\ Inform.\ Theory 8 (1962) 200--202.

\bibitem{HaasCTC:17}
J.\ I.\ Haas, J.\ Cahill, J.\ Tremain, P.\ G.\ Casazza,
Constructions of biangular tight frames and their relationships with equiangular tight frames,
Available online: arXiv:1703.01786

\bibitem{HaasHM:17}
J.\ I.\ Haas, N.\ Hammen, D.\ G.\ Mixon,
The Levenstein bound for packings in projective spaces,
Wavelets and Sparsity XVII (2017) 103940V.

\bibitem{HolmesP:04}
R.\ B.\ Holmes, V.\ I.\ Paulsen,
Optimal frames for erasures,
Linear Algebra Appl.\ 377 (2004) 31--51.

\bibitem{HurdS:07}
S.\ P.\ Hurd, D.\ G.\ Sarvate,
Balanced Ternary Designs,
in:\ Handbook of Combinatorial Designs, 2nd ed.,
C.\ J.\ Colbourn, J.\ H.\ Dinitz, eds., CRC Press, 2007, pp.\ 330--333.

\bibitem{IversonJM:17}
J.\ W.\ Iverson, J.\ Jasper, D.\ G.\ Mixon,
Optimal line packings from finite group actions,
Available online: arXiv:1709.03558

\bibitem{IversonJM:16}
J.\ W.\ Iverson, J.\ Jasper, D.\ G.\ Mixon,
Optimal line packings from nonabelian groups,
Available online: arXiv:1609.09836

\bibitem{JasperMF:13}
J.\ Jasper, D.\ G.\ Mixon, M.\ Fickus,
Kirkman equiangular tight frames and codes,
IEEE Trans.\ Inform.\ Theory 60 (2013) 170--181.

\bibitem{Levenshtein:92}
V.\ I.\ Levenshtein,
Designs as maximum codes in polynomial metric spaces.,
Acta Appl.\ Math.\ 29 (1992) 1--82.

\bibitem{LustigDSP:08}
M.\ Lustig, D.\ L.\ Donoho, J.\ M.\ Santos, J.\ M.\ Pauly,
Compressed sensing MRI,
IEEE Signal Process.\ Mag.\ 25 (2008) 72--82.

\bibitem{MathonR:07}
R.\ Mathon, A.\ Rosa,
2-$(v,k,\lambda)$ Designs of Small Order,
in:\ Handbook of Combinatorial Designs, 2nd ed.,
C.\ J.\ Colbourn, J.\ H.\ Dinitz, eds., CRC Press, 2007, pp.\ 330--333.

\bibitem{Mixon:online}
D.\ G.\ Mixon,
Conjectures from SampTA,
Short, Fat Matrices (weblog),
\url{https://dustingmixon.wordpress.com/2015/07/08/conjectures-from-sampta/}

\bibitem{MixonQKF:13}
D.\ G.\ Mixon, C.\ J.\ Quinn, N.\ Kiyavash, M.\ Fickus,
Fingerprinting with equiangular tight frames,
IEEE Trans.\ Inf.\ Theory 59 (2013) 1855--1865.

\bibitem{MusinT:15}
O.\ R.\ Musin, A.\ S.\ Tarasov,
The Tammes problem for $N= 14$,
Exp.\ Math.\ 24 (2015) 460--468.

\bibitem{JPEG2000}
Overview of JPEG 2000,
\url{https://jpeg.org/jpeg2000/index.html}

\bibitem{Rankin:56}
R.\ A.\ Rankin,
On the minimal points of positive definite quadratic forms,
Mathematika 3 (1956) 15--24.

\bibitem{Rankin:55}
R.\ A.\ Rankin,
The closest packing of spherical caps in $n$ dimensions,
In:\ Proceedings of the Glasgow Mathematical Association, vol.\ 2, Cambridge University Press, 1955, pp.\ 139--144.

\bibitem{RenesBSC:04}
J.\ M.\ Renes, R.\ Blume-Kohout, A.\ J.\ Scott, C.\ M.\ Caves,
Symmetric informationally complete quantum measurements,
J.\ Math.\ Phys.\ 45 (2004) 2171--2180.

\bibitem{Schlafli:49}
L.\ Schl\"{a}fli,
Theorie der vielfachen Kontinuit\"{a}t,
Collected mathematical works (in German),
Birkh\"{a}user Verlag, 1949.

\bibitem{Schmidt:76}
W.\ M.\ Schmidt,
Equations over Finite Fields:\ An Elementary Approach,
Springer, 1976.

\bibitem{Seidel:73}
J.\ J.\ Seidel,
A survey of two-graphs,
in:\ Proc.\ Intern.\ Coll.\ Teorie Combinatorie, 1973, 481--511.

\bibitem{Sloane:online}
N.\ J.\ A.\ Sloane,
Packings in Grassmannian spaces,
\url{http://neilsloane.com/grass/}

\bibitem{Sloane:online2}
N.\ J.\ A.\ Sloane,
Spherical Codes,
\url{http://neilsloane.com/packings/}

\bibitem{StrohmerH:03}
T.\ Strohmer, R.\ W.\ Heath,
Grassmannian frames with applications to coding and communication,
Appl.\ Comput.\ Harmon.\ Anal.\ 14 (2003) 257--275.

\bibitem{SustikTDH:07}
M.\ A.\ Sustik, J.\ A.\ Tropp, I.\ S.\ Dhillon, R.\ W.\ Heath,
On the existence of equiangular tight frames,
Linear Algebra Appl.\ 426 (2007) 619--635.

\bibitem{Szollosi:14}
F.\ Sz\"{o}ll\H{o}si,
All complex equiangular tight frames in dimension $3$,
Available online: arXiv:1402.6429

\bibitem{Toth:65}
L.\ T\'{o}th,
Distribution of points in the elliptic plane,
Acta Math.\ Hung.\ 16 (1965) 437--440.

\bibitem{Tremain:09}
J.\ C.\ Tremain,
Concrete constructions of real equiangular line sets,
Available online: arXiv:0811.2779

\bibitem{Tropp:08}
J.\ A.\ Tropp,
On the conditioning of random subdictionaries,
Appl.\ Comput.\ Harmon.\ Anal.\ 25 (2008) 1--24.

\bibitem{Waldron:09}
S.\ Waldron,
On the construction of equiangular frames from graphs,
Linear Algebra Appl.\ 431 (2009) 2228--2242.

\bibitem{Welch:74}
L.\ R.\ Welch,
Lower bounds on the maximum cross correlation of signals,
IEEE Trans.\ Inform.\ Theory 20 (1974) 397--399.

\bibitem{XiaZG:05}
P.\ Xia, Z.\ Shengli, G.\ B.\ Giannakis,
Achieving the Welch bound with difference sets,
IEEE Trans.\ Inform.\ Theory 51 (2005) 1900--1907.

\bibitem{Zauner:99}
G.\ Zauner,
Quantendesigns - Grundz\"{u}ge einer nichtkommutativen Designtheorie,
Ph.D. thesis, U.\ Vienna, 1999.

\end{thebibliography}
\end{document}